\documentclass[a4paper]{article}
\usepackage{a4wide}
\usepackage[UKenglish]{babel}
\usepackage{amsmath,amssymb,amsthm,mathtools,empheq,mathrsfs,bbold}
\usepackage{bbm}
\usepackage{graphicx,subfigure}
    \graphicspath{{./}{fig/}}
\usepackage{xcolor}
\usepackage[colorlinks=true,linkcolor=blue,citecolor=red]{hyperref}
\usepackage{authblk}
\usepackage[inline]{enumitem}

\DeclarePairedDelimiter\abs{\lvert}{\rvert}
\DeclarePairedDelimiter\ave{\langle}{\rangle}
\newcommand{\rmI}{\mathrm{I}}
\newcommand{\cM}{\mathcal{M}}
\newcommand{\norm}[2]{\Vert#1\Vert_{#2}}
\newcommand{\R}{\mathbb{R}}
\newcommand{\rmS}{\mathrm{S}}
\newcommand{\T}{\mathcal{T}}
\newcommand*\tcircle[1]{%
  \raisebox{-0.5pt}{%
    \textcircled{\fontsize{7pt}{0}\fontfamily{phv}\selectfont #1}%
  }%
}

\newcommand{\rev}[1]{#1}

\newtheorem{proposition}{Proposition}[section]
\newtheorem{theorem}[proposition]{Theorem}
\newtheorem{lemma}[proposition]{Lemma}
\theoremstyle{remark}\newtheorem{remark}[proposition]{Remark}

\allowdisplaybreaks

\title{Derivation and quasi-invariant asymptotics of phenotype-structured integro-differential models\footnote{E.B., T.L., and A.T. are members of INdAM-GNFM. T.L. and A.T. gratefully acknowledge support from the Italian Ministry of University and Research (MUR) through the grant PRIN2022-PNRR project (No. P2022Z7ZAJ) ``A Unitary Mathematical Framework for Modelling Muscular Dystrophies'' (CUP: E53D23018070001) funded by the European Union NextGenerationEU.}}

\author{Emanuele Bernardi}
\author{Tommaso Lorenzi}
\author{Andrea Tosin}
\affil{{\small Department of Mathematical Sciences ``G. L. Lagrange'' \\ Politecnico di Torino, Italy}}
\date{}

\begin{document}
\maketitle

\begin{abstract}
\rev{Building upon kinetic theory approaches for mass-varying multi-agent systems}, we develop a modelling framework for phenotype-structured populations that makes it possible to bridge individual-level mechanisms with population-scale evolutionary dynamics. We start by formulating a stochastic agent-based model, which describes the dynamics of single population members undergoing proliferation, death, and phenotype changes. Then, we formally derive the corresponding mesoscopic model, which consists of an integro-differential equation for the distribution of population members over the space of phenotypes, where phenotype changes are modelled via an integral kernel. Finally, considering a quasi-invariant regime of small but frequent phenotype changes, we rigorously derive a non-local Fokker--Planck-type equation counterpart of this model, wherein phenotype changes are taken into account by an advection-diffusion term. The theoretical results obtained are illustrated through a sample of results of numerical simulations.

\medskip


\medskip

\end{abstract}

\section{Introduction}
Starting from the pioneering work of Ludwig Boltzmann~\cite{boltzmann1970CHAPTER}, who aimed to explain complex thermodynamic phenomena out of simple elastic collisions among gas molecules, kinetic theory evolved in recent years into a conceptual framework to model multi-agent systems~\cite{pareschi2013BOOK}, that is, systems composed of a large number of interacting entities (i.e. agents) that, as a result of their interactions, give rise to non-trivial collective behaviours. To name just a few representative application domains in which this framework has fruitfully been employed, we mention the study of socio-economic phenomena, such as opinion formation~\cite{boudin2016PHYSA,duering2015PRSA,toscani2006CMS} and wealth distribution~\cite{bisi2017BUMI,cordier2005JSP,duering2009RMUP}, vehicular traffic~\cite{borsche2022PHYSA,prigogine1960OR,tosin2019MMS}, and, more recently, the spread of infectious diseases~\cite{bisi2024JNS,dimarco2020PRE,loy2021MBE}. Moreover, kinetic equations have also been employed to investigate the theoretical properties of particle-based algorithms for non-convex optimisation~\cite{benfenati2022AMO,carrillo2018M3AS}.

The multiscale nature of the kinetic paradigm is particularly suited to the aforementioned application domains, since it enables one to link rigorously, in a probabilistic-statistical manner, the individual dynamics of the state of single agents to the aggregate trends that they generate at the level of the whole agent population. A common feature of such application domains, however, is that, for the aspects of interest, the total mass of the system (i.e. the number of agents that the system comprises) can be assumed to be constant. This leads to mass conservation, a principle that is also a pillar of the classical kinetic theory, as it facilitates operating in a probabilistic setting. Conversely, applications of the kinetic paradigm to scenarios where non-conservative phenomena occur, and thus \rev{the mass may change in time}, have so far been mainly \rev{focused on} gas mixtures possibly undergoing chemical reactions~\cite{andries2002JSP,bisi2010PRE,borsoni2022CMP,groppi2004JSP,groppi1999JMC}, multi-agent systems featuring label switching dynamics~\cite{bisi2024PHYSD,loy2021KRM}, \rev{socio-economic systems~\cite{bisi2017BUMI}, moving cells and animals~\cite{engwer2015glioma,hillen1,othmer1,perthame3}, and epidemiological systems~\cite{bisi2024JNS,dellamarca1,dellamarca2,dimarco1,zanella1}}.

Non-conservative phenomena are central to the study of the evolution of observable traits (i.e. phenotypes) in living systems. In fact, since proliferation and death are key drivers of phenotypic evolution via natural selection, variation in the number of population members is intrinsic to evolutionary dynamics of populations structured by phenotypic traits. Thus, \rev{the size of the population} is typically not conserved. This presents a challenge when deriving deterministic continuum models for the evolutionary dynamics of phenotype-structured populations from underlying stochastic agent-based models, which track proliferation, death, and phenotype changes of single population members. Such a challenge has so far been tackled through probabilistic methods~\cite{champagnat2005individual,champagnat2006unifying}, which make it possible to rigorously derive population-level models as the limit of corresponding agent-based models when the number of population members tends to infinity, and formal limiting procedures~\cite{ardavseva2020comparative,chisholm2016evolutionary}, which enable one to obtain the population-scale counterparts of on-lattice agent-based models when the lattice parameters tend to zero. 

In this paper, building upon kinetic theory approaches for multi-agent systems of the type advanced and adopted in~\cite{bernardi2025SAM,lorenzi2023modelling,loy2021KRM}, and generalising them to \rev{mass-varying multi-agent systems}, we develop a modelling framework for phenotype-structured populations that makes it possible to bridge individual-level mechanisms with population-scale evolutionary dynamics. 

We start by formulating a stochastic agent-based model, which describes the dynamics of single population members that undergo proliferation, death, and phenotype changes. Then, we formally derive the corresponding mesoscopic model, which consists of an integro-differential equation (IDE) for the distribution of population members over the space of phenotypes (i.e. the phenotype distribution function), where phenotype changes are modelled via an integral kernel. Finally, considering a quasi-invariant regime of small but frequent phenotype changes, we rigorously derive a non-local partial differential equation (PDE) counterpart of this model, wherein phenotype changes are taken into account by an advection-diffusion term (i.e. a non-local Fokker–Planck-type equation). For a summary of possible applications and qualitative properties of the solutions of cognate IDEs and non-local PDEs, we refer the interested reader to Sections 3.1 and 4.2 of the recent review~\cite{lorenzi2025phenotype} and references therein.

In contrast to the limiting approach employed in~\cite{ardavseva2020comparative,chisholm2016evolutionary}, which directly leads to a non-local PDE model of evolutionary dynamics analogous to the one considered here as the population-scale counterpart of an agent-based model, our approach allows one to first obtain an IDE and then derive the non-local PDE as the corresponding quasi-invariant limit. This offers a unitary method to bridge alternative representations of the evolutionary dynamics of phenotype-structured populations -- from agent-based models through to IDE models to non-local PDEs. Furthermore, it makes the range of applications of our approach considerably wider, as the integral kernels comprised in IDE models permit a finer representation of the mechanisms underlying phenotype changes than the advection-diffusion terms contained in non-local PDE models. In addition, while non-local diffusion models for the evolutionary dynamics of phenotype-structured populations are usually formally derived from the corresponding IDE counterparts (see, for instance~\cite{hamel2020dynamics}, and references therein), here we present a rigorous derivation under assumptions that encompass a wide spectrum of biologically relevant scenarios and lead to a non-local reaction-advection-diffusion model.

The rest of the paper is organised as follows. In Section~\ref{sect:agentbased}, we formulate the agent-based model. In Section~\ref{sect:mesoscopic}, a formal derivation of the corresponding IDE model is carried out. In Section~\ref{sect:quasi-invariant}, the non-local PDE analogue of this model is rigorously derived in the limit of small but frequent phenotype changes. In Section~\ref{sect:num}, we present a sample of results of numerical simulations through which we illustrate the theoretical results obtained in the previous sections. In Section~\ref{sect:conclusions}, we conclude with a summary of key findings and a discussion of future directions. 

\section{Stochastic agent-based model}
\label{sect:agentbased}
We consider a population that is structured by a continuous variable $v\in\R$, which models the individuals’ phenotype and captures variability in the rates of proliferation and death of the population members. Extending the structured compartmental modelling approach we developed in~\cite{bernardi2025SAM,lorenzi2023modelling,loy2021KRM} to scenarios where, in addition to undergoing phenotype changes, individuals may also proliferate and die, we look at the population as a compartment labelled by the index $i=1$ and structured by the variable $v$. Moreover, we introduce an additional compartment, labelled by the index $i=0$, wherein all individuals share the same phenotype, which is modelled by a constant $v_0\in\R$.  As illustrated by the schematic in Figure~\ref{fig:intro}, we conceptualise the proliferation and death of population members as \textit{compartment switching} processes, whilst phenotype changes undergone by population members are conceptualised as \textit{structuring-variable switching} processes. 

\begin{figure}[t]
    \centering
    \includegraphics[width=0.7\linewidth]{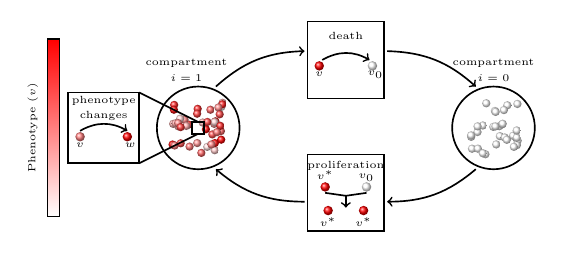}	
    \caption{Schematic illustrating the conceptual framework underlying the agent-based model}
    \label{fig:intro}
\end{figure}

In more detail, between time $t\in [0,\,+\infty)$ and time $t+\Delta{t}$, with $\Delta{t}>0$:
\begin{enumerate}[label=(\Roman*)]
\item The \textit{proliferation} of a population member is modelled as an \textit{interaction-driven compartment switching}, whereby the interaction between a focal individual in compartment $i=0$ and an individual in compartment $i^*=1$ with phenotype $v^*$ leads the focal individual to transition into compartment $i^*=1$ acquiring the phenotype $v^*$ at rate $p(v^*)$. This results in a new individual with phenotype $v^*$ being introduced in the population. \label{def:proliferation}
\item The \textit{death} of a population member is modelled as a \textit{spontaneous compartment switching} (i.e. compartment switching not driven by interactions with other individuals), whereby a focal individual in compartment $i=1$ with phenotype $v$ spontaneously transitions into compartment $i=0$ acquiring the phenotype $v_0$ at rate $d(v)$. This results in an individual with phenotype $v$ being removed from the population. \label{def:death}
\item \textit{Phenotype changes} of population members are modelled as \textit{structuring-variable switching} processes, whereby, at rate $\mu \in\R_+$, a focal individual in compartment $i=1$ switches from phenotype $v$ to phenotype $w$ with probability $M(w\vert v)$. \label{def:phen_changes}
\end{enumerate}
The functions $p$ and $d$ and the kernel $M$ satisfy the following assumptions
\begin{equation}
    p,\,d\in C(\R), \qquad p,\,d:\R\to\R_+, \qquad M\in\mathscr{P}(\R;\,C(\R)),
    \label{ass:pdM}
\end{equation}
where $\mathscr{P}(\R)$ denotes the set of probability distributions defined on the measurable space $(\R,\,\mathcal{B}(\R))$, with $\mathcal{B}(\R)$ the Borel $\sigma$-algebra of $\R$.

In this conceptual framework, individuals are regarded as indistinguishable agents and their microscopic state at time $t$ is described by the random vector $(I_{t},\,V_{t})\in\{0,1\}\times\R$. The discrete random variable $I_{t}$ specifies the compartment of an agent, while the continuous random variable $V_{t}$ models the agent's phenotype at time $t$. We let
\begin{equation}
    (I_{t},\,V_{t})\sim f(i,v,t)
    \label{eq:f}
\end{equation}
and we introduce the notation
\begin{equation}
    \rho_i(t):=\int_\R f(i,v,t)\,dv.
    \label{eq:rho}
\end{equation}
The distribution function $f:\{0,1\}\times\R\times [0,\,+\infty)\to\R_+$, being a probability density function, is such that
\begin{equation}
    \sum_{i=0}^1\int_\R f(i,v,t)\,dv=\sum_{i=0}^1\rho_i(t)=1, \qquad \forall\,t\geq 0.
    \label{eq:globalm}
\end{equation}

Moreover, in line with~\ref{def:proliferation}--\ref{def:phen_changes}, we let the microscopic states of the agents evolve in time due to \textit{interaction-driven compartment switching}, \textit{spontaneous compartment switching}, and \textit{structuring-variable switching}. \rev{Regarding such switching phenomena as statistically independent, and taking $\Delta{t}$ sufficiently small so that subsequent interaction-driven and spontaneous compartment switching between time $t$ and $t+\Delta{t}$ can be, to a first approximation, neglected,} we model the evolution of the microscopic state $(I_{t},\,V_{t})$ of a focal agent through the following system
\begin{equation}
    \begin{cases}
        I_{t+\Delta{t}}=(1-\rmS^\rmI)(1-\rmS^{\rmI\rmI})I_t+\rmS^\rmI(1-\rmS^{\rmI\rmI})I_t'+(1-\rmS^\rmI)\rmS^{\rmI\rmI}I_t'', \\[3mm]
        V_{t+\Delta{t}}=(1-\rmS^\rmI)(1-\rmS^{\rmI\rmI})\left[(1-\rmS^{\rmI\rmI\rmI})V_t+\rmS^{\rmI\rmI\rmI}W_t\right]+\rmS^\rmI(1-\rmS^{\rmI\rmI})V_t' \\
        \phantom{V_{t+\Delta{t}}=} +(1-\rmS^\rmI)\rmS^{\rmI\rmI}V_t''.
    \end{cases}
    \label{eq:3rand}
\end{equation}
In system~\eqref{eq:3rand}, $\rmS^\rmI,\,\rmS^{\rmI\rmI},\,\rmS^{\rmI\rmI\rmI}\in\{0,\,1\}$ are Bernoulli random variables, \rev{that are assumed to be independent of both one another and of the other variables $I_t$, $I'_t$, $I''_t$, $V_t$, $V'_t$, $V''_t$, $W_t$, and are also} such that
\begin{align}
    \begin{aligned}[c]
        & \operatorname{Prob}{(\rmS^\rmI=1\vert I_t=i,\,V_t=v,\,I_t^\ast=i^\ast,\,V_t^\ast=v^\ast)}=\frac{s^\rmI_{i,i^\ast}(v,v^\ast)\Delta{t}}{1+s^\rmI_{i,i^\ast}(v,v^\ast)\Delta{t}}, \\
        & \operatorname{Prob}{(\rmS^{\rmI\rmI}=1\vert I_t=i,\,V_t=v)}=\frac{s^{\rmI\rmI}_i(v)\Delta{t}}{1+s^{\rmI\rmI}_i(v)\Delta{t}}, \\
        & \operatorname{Prob}{(\rmS^{\rmI\rmI\rmI}=1\vert I_t=i)}=\frac{s^{\rmI\rmI\rmI}_i\Delta{t}}{1+s^{\rmI\rmI\rmI}_i\Delta{t}}.
    \end{aligned}	
    \label{eq:Prob}
\end{align}	
Furthermore, the pairs of random variables
\begin{equation}
    (I'_t,\,V_t'\vert I_t^\ast,\,V_t^\ast)\sim\T^\rmI(i',v'\vert I_t^\ast,\,V_t^\ast) \qquad \text{and} \qquad
        (I''_t,\,V''_t)\sim\T^{\rmI\rmI}(i'',\,v'')
    \label{distrib:TITII}
\end{equation}
model, respectively, the new microscopic state of the focal agent when interaction-driven compartment switching and spontaneous compartment switching occurs, while the random variable 
\begin{equation}
    (W_t\vert V_t)\sim\T^{\rmI\rmI\rmI}(w\vert V_t)
    \label{distrib:TIII}
\end{equation}
models the new phenotype of the focal agent when structuring-variable switching takes place. \rev{Note also that, under~\eqref{eq:Prob}, subsequent interaction-driven and spontaneous compartment switching occur with probability of order higher than $\Delta{t}$, consistently with the first approximation assumption underlying system~\eqref{eq:3rand}.} Moreover, $s^\rmI_{i,i^\ast}(v,v^\ast)\geq 0$ is the rate of interaction-driven compartment switching, $s^{\rmI\rmI}_i(v)\geq 0$ is the rate of spontaneous compartment switching and $s^{\rmI\rmI\rmI}_i\geq 0$ is the rate of structuring-variable switching. In particular, to incorporate~\ref{def:proliferation}--\ref{def:phen_changes} into the model, we let 
\begin{subequations}
\begin{empheq}[]{align}
	& s^\rmI_{i,i^\ast}(v,v^\ast)\equiv s^\rmI_{i,i^\ast}(v^\ast):=\delta_{i,0}\delta_{i^\ast,1}p(v^\ast), \label{def:sI} \\[3mm]
	& s^{\rmI\rmI}_i(v):=\delta_{i,1}d(v), \label{def:sII} \\[3mm]
	& s^{\rmI\rmI\rmI}_i:=\delta_{i,1}\mu, \label{def:sIII}
\end{empheq}
\end{subequations}
and
\begin{subequations}
\begin{empheq}[]{align}
	& \T^\rmI(i',v'\vert i^\ast,\,v^\ast):=\delta_{i',i^\ast}\delta_{v^\ast}(v'), \label{def:mathcalTI} \\[3mm]
	& \T^{\rmI\rmI}(i'',v''):=\delta_{i'',0}\delta_{v_0}(v''), \label{def:mathcalTII} \\[3mm]
	& \T^{\rmI\rmI\rmI}(w\vert v):=M(w\vert v), \label{def:mathcalTIII}
\end{empheq}
\end{subequations}
where $\delta_{i,j}$ is the Kronecker delta and $\delta_{x}(y)$ is the Dirac delta centred at $y=x$.

\rev{As a final remark, we note that all model terms considered here are intended to be dimensionless.}

\section{Mesoscopic model}
\label{sect:mesoscopic}
Starting from the system~\eqref{eq:3rand}, we now formally derive the mesoscopic model corresponding to the agent-based model presented in the previous section, which consists of an evolution equation for the distribution of population members over the phenotype domain.

\paragraph{General evolution equation for expectations of observables} We start by noting that, since the components of the pair $(I_{t+\Delta{t}},\,V_{t+\Delta t})$ are given by~\eqref{eq:3rand}, for any observable $\Phi:\{0,\,1\}\times\R\to\R$, namely an arbitrary test function defined for $(i,\,v)\in\{0,\,1\}\times\R$, the expectation
$$ \ave{\Phi(I_t,V_t)}:=\sum_{i=0}^1\int_\R\Phi(i,v)f(i,v,t)\,dv $$
satisfies (see, for instance,~\cite{pareschi2013BOOK})
\begin{align*}
    \resizebox{\textwidth}{!}{$\displaystyle
    \begin{aligned}
    \ave*{\Phi(I_{t+\Delta{t}},V_{t+\Delta{t}})} &= \ave*{\left(1-\frac{s^\rmI_{I_t,I^\ast_t}(V_t,V^\ast_t)\Delta{t}}{1+s^\rmI_{I_t,I^\ast_t}(V_t,V^\ast_t)\Delta{t}}\right)
        \left(1-\frac{s^{\rmI\rmI}_{I_t}(V_t)\Delta{t}}{1+s^{\rmI\rmI}_{I_t}(V_t)\Delta{t}}\right)\left(1-\frac{s^{\rmI\rmI\rmI}_{I_t}\Delta{t}}{1+s^{\rmI\rmI\rmI}_{I_t}\Delta{t}}\right)\Phi(I_t,V_t)} \\
    &\phantom{=} +\ave*{\left(1-\frac{s^\rmI_{I_t,I^\ast_t}(V_t,V^\ast_t)\Delta{t}}{1+s^\rmI_{I_t,I^\ast_t}(V_t,V^\ast_t)\Delta{t}}\right)
        \left(1-\frac{s^{\rmI\rmI}_{I_t}(V_t)\Delta{t}}{1+s^{\rmI\rmI}_{I_t}(V_t)\Delta{t}}\right)\frac{s^{\rmI\rmI\rmI}_{I_t}\Delta{t}}{1+s^{\rmI\rmI\rmI}_{I_t}\Delta{t}}
            \Phi(I_t,W_t)} \\
    &\phantom{=} +\ave*{\frac{s^\rmI_{I_t,I^\ast_t}(V_t,V^\ast_t)\Delta{t}}{1+s^\rmI_{I_t,I^\ast_t}(V_t,V^\ast_t)\Delta{t}}
        \left(1-\frac{s^{\rmI\rmI}_{I_t}(V_t)\Delta{t}}{1+s^{\rmI\rmI}_{I_t}(V_t)\Delta{t}}\right)\Phi(I'_t,V'_t)} \\
    &\phantom{=} +\ave*{\left(1-\frac{s^\rmI_{I_t,I^\ast_t}(V_t,V^\ast_t)\Delta{t}}{1+s^\rmI_{I_t,I^\ast_t}(V_t,V^\ast_t)\Delta{t}}\right)
        \frac{s^{\rmI\rmI}_{I_t}(V_t)\Delta{t}}{1+s^{\rmI\rmI}_{I_t}(V_t)\Delta{t}}\Phi(I''_t,V''_t)}.
    \end{aligned}
    $}
\end{align*}
\rev{We stress that, owing to the assumed independence of the various random variables, the notation $\ave{\cdot}$ consistently denotes expectations that can be computed recursively.} From the above equation, rearranging terms, dividing through by $\Delta{t}$ and letting $\Delta{t}\to 0^+$, we formally obtain the following evolution equation:
\begin{align}
    \begin{aligned}[b]
        \frac{d}{dt}\ave{\Phi(I_t,V_t)} &= \ave{s^\rmI_{I_t,I^\ast_t}(V_t,V^\ast_t)\left(\Phi(I_t',V_t')-\Phi(I_t,V_t)\right)} \\
        &\phantom{=} +\ave{s^{\rmI\rmI}_{I_t}(V_t)\left(\Phi(I_t'',V_t'')-\Phi(I_t,V_t)\right)} \\
        &\phantom{=} +\ave{s^{\rmI\rmI\rmI}_{I_t}\left(\Phi(I_t,W_t)-\Phi(I_t,V_t)\right)}.
    \end{aligned}
    \label{eq:8varphi}
\end{align}
Expressing the expectations $\ave{\cdot}$ in~\eqref{eq:8varphi} in terms of sums and integrals against the probability density function $f$, recalling~\eqref{distrib:TITII} and~\eqref{distrib:TIII}, and introducing the more compact notation $f_i(v,t):= f(i,v,t)$, the evolution equation~\eqref{eq:8varphi} can then be rewritten as
\begin{align}
    \resizebox{.93\textwidth}{!}{$\displaystyle
    \begin{aligned}[b]
        \frac{d}{dt}  &\sum_{i=0}^1\int_\R\Phi(i,v)f_i(v,t)\,dv= \\
        &= \underbrace{\sum_{i=0}^1\int_\R\Phi(i,v)\left(\sum_{i'=0}^1\sum_{i^\ast=0}^1\int_\R\int_\R s^\rmI_{i',i^\ast}(v',v^\ast)\T^\rmI(i,v\vert i^\ast,\,v^\ast)
            f_{i'}(v',t)f_{i^\ast}(v^\ast,t)\,dv'\,dv^\ast\right)dv}_{=:\tcircle{i}} \\
        &\phantom{=} -\underbrace{\sum_{i=0}^1\int_\R\Phi(i,v)\left(\sum_{i^\ast=0}^1\int_\R s^\rmI_{i,i^\ast}(v,v^\ast)f_{i^\ast}(v^\ast,t)\,dv^\ast\right)f_i(v,t)\,dv}_{=:\tcircle{ii}} \\ 
        &\phantom{=} +\underbrace{\sum_{i=0}^1\int_\R\Phi(i,v)\left(\sum_{i''=0}^1\int_\R s^{\rmI\rmI}_{i''}(v'')\T^{\rmI\rmI}(i,v)
            f_{i''}(v'',t)\,dv''\right)dv}_{=:\tcircle{iii}} \\
        &\phantom{=} -\underbrace{\sum_{i=0}^1\int_\R\Phi(i,v)s^{\rmI\rmI}_i(v)f_i(v,t)\,dv}_{=:\tcircle{iv}} \\
        &\phantom{=} +\underbrace{\sum_{i=0}^1\int_\R\Phi(i,v)\left(\int_\R s^{\rmI\rmI\rmI}_i\T^{\rmI\rmI\rmI}(v\vert w)f_i(w,t)\,dw\right)dv}_{=:\tcircle{v}} \\
        &\phantom{=} -\underbrace{\sum_{i=0}^1\int_\R\Phi(i,v)s^{\rmI\rmI\rmI}_if_i(v,t)\,dv}_{=:\tcircle{vi}}.
    \end{aligned}
    $}
    \label{eq:generic_form}
\end{align}

From~\eqref{eq:generic_form} we now derive the evolution equations for $f_0(v,t)$ and $f_1(v,t)$, from which we shall then derive an evolution equation for the distribution of population members over the phenotype domain, that is, the mesoscopic counterpart of the agent-based model. To do so, we choose 
\begin{equation}
    \Phi(i,v):=\delta_{i,l}\varphi(v),
    \label{eq:test_fun}
\end{equation}
where $l\in\{0,\,1\}$ and $\varphi$ is an observable quantity (test function) depending only on $v$. For ease of presentation, we carry out calculations for the left-hand side (LHS) of~\eqref{eq:generic_form} and the terms \tcircle{i}-\tcircle{vi} on the right-hand side (RHS) of~\eqref{eq:generic_form} one by one, first for $l=0$ and then for $l=1$.

\paragraph{Evolution equation for~\texorpdfstring{$\boldsymbol{f_0}$}{}} First, substituting~\eqref{eq:test_fun} with $l=0$ into the LHS of~\eqref{eq:generic_form} yields
\begin{equation}
    \frac{d}{dt}\sum_{i=0}^1\int_\R\Phi(i,v)f_i(v,t)\,dv=\int_\R\varphi(v)\partial_tf_0(v,t)\,dv.
    \label{eq:lhsterm0}
\end{equation}

Then, substituting~\eqref{eq:test_fun} with $l=0$ along with~\eqref{def:sI} and~\eqref{def:mathcalTI} into \tcircle{i} and \tcircle{ii} on the RHS of~\eqref{eq:generic_form} yields 
\begin{align}
    \begin{aligned}[b]
        \tcircle{i}+\tcircle{ii} &= \int_\R\varphi(v)\left(\sum_{i'=0}^1\sum_{i^\ast=0}^1\int_\R\int_\R s^\rmI_{i',i^\ast}(v',v^\ast)\T^\rmI(0,v\vert i^\ast,\,v^\ast)
            f_{i'}(v',t)f_{i^\ast}(v^\ast,t)\,dv'\,dv^\ast\right)dv \\
        &\phantom{=} -\int_\R\varphi(v)\left(\sum_{i^\ast=0}^1\int_\R s^\rmI_{0,i^\ast}(v,v^\ast)f_{i^\ast}(v^\ast,t)\,dv^\ast\right)f_0(v,t)\,dv \\ 
        &= -\int_\R\varphi(v)\left(\int_\R p(v^\ast)f_1(v^\ast,t)\,dv^\ast\right)f_0(v,t)\,dv. 
    \end{aligned}
    \label{eq:rhsI10}
\end{align}

Next, substituting~\eqref{eq:test_fun} with $l=0$ along with~\eqref{def:sII} and~\eqref{def:mathcalTII} into \tcircle{iii} and \tcircle{iv} on the RHS of~\eqref{eq:generic_form} yields 
\begin{align}
    \begin{aligned}[b]
        \tcircle{iii}+\tcircle{iv} &= \int_\R\varphi(v)\left(\sum_{i''=0}^1\int_\R s^{\rmI\rmI}_{i''}(v'')\T^{\rmI\rmI}(0,v)
            f_{i^{\prime \prime}}(v'',t)\,dv''\right)dv \\
        &\phantom{=} -\int_\R\varphi(v)s^{\rmI\rmI}_0(v)f_0(v,t)\,dv \\
        &= \int_\R\varphi(v)\left(\int_\R d(v'')f_1(v'',t)\,dv''\right)\delta_{v^0}(v)\,dv.
    \end{aligned}
    \label{eq:rhsI20}
\end{align}

Finally, substituting~\eqref{eq:test_fun} with $l=0$ into \tcircle{v} and \tcircle{vi} on the RHS of~\eqref{eq:generic_form} and using the fact that $s^{\rm III}_{0}=0$ (cf. definition~\eqref{def:sIII}) yields 
\begin{equation}
    \tcircle{v}+\tcircle{vi}=\int_\R\varphi(v)\left(\int_\R s^{\rmI\rmI\rmI}_{0}\T^{\rmI\rmI\rmI}(v\vert w)f_0(w,t)\,dw\right)dv
        -\int_\R\varphi(v)s^{\rmI\rmI\rmI}_0f_0(v,t)\,dv=0.
    \label{eq:rhsI30}
\end{equation}

Combining~\eqref{eq:lhsterm0}--\eqref{eq:rhsI30} and rearranging terms we obtain, owing to the arbitrariness of $\varphi$, the following evolution equation for $f_0$:
\begin{equation}
    \partial_tf_0(v,t)=\int_\R d(v'')f_1(v'',t)\,dv''\delta_{v^0}(v)
        -\left(\int_\R p(v^\ast)f_1(v^\ast,t)\,dv^\ast\right)f_0(v,t).
    \label{eq:wfeqf0}
\end{equation}
Note that integrating the differential equation~\eqref{eq:wfeqf0} over $\R$ and recalling~\eqref{eq:rho} we find
$$ \frac{d\rho_0}{dt}=\int_\R d(v'')f_1(v'',t)\,dv''
    -\left(\int_\R p(v^\ast)f_1(v^\ast,t)\,dv^\ast\right)\rho_0, $$
that is,
\begin{align}
    \begin{aligned}[b]
        \rho_0(t) &= \rho_0(0)e^{-\int_0^t\int_\R p(v^\ast)f_1(v^\ast,s)\,dv^\ast\,ds} \\
        &\phantom{=} +\int_0^t\left(\int_\R d(v'')f_1(v'',s)\,dv''\right)
            e^{-\int_s^t\int_\R p(v^\ast)f_1(v^\ast,\tau)\,dv^\ast\,d\tau}\,ds.
    \end{aligned}
    \label{eq:rho0_semiexp}
\end{align}
Moreover, solving the differential equation~\eqref{eq:wfeqf0} subject to the initial condition $f_0(v,0)=\rho_0(0)\delta_{v^0}(v)$ and substituting~\eqref{eq:rho0_semiexp} into the expression for $f_0$ so obtained yields $f_0(v,t) = \rho_0(t) \delta_{v^0}(v)$ for all $t \geq 0$, consistently with the fact that, by construction, all individuals in the compartment $i=0$ are expected to share the same phenotype $v^0$ at all times.

\paragraph{Evolution equation for~\texorpdfstring{$\boldsymbol{f_1}$}{}} First, substituting~\eqref{eq:test_fun} with $l=1$ into the LHS of~\eqref{eq:generic_form} yields
\begin{equation}
    \frac{d}{dt}\sum_{i=0}^1\int_\R\Phi(i,v)f_i(v,t)\,dv=\int_\R\varphi(v)\partial_tf_1(v,t)\,dv.
    \label{eq:lhsterm1}
\end{equation}

Then, substituting~\eqref{eq:test_fun} with $l=1$ along with~\eqref{def:sI} and~\eqref{def:mathcalTI} into \tcircle{i} and \tcircle{ii} on the RHS of~\eqref{eq:generic_form}, and recalling~\eqref{eq:rho}, yields 
\begin{align}
    \begin{aligned}[b]
        \tcircle{i}+\tcircle{ii} &= \int_\R\varphi(v)\left(\sum_{i'=0}^1\sum_{i^\ast=0}^1\int_\R\int_\R s^\rmI_{i',i^\ast}(v',v^\ast)\T^\rmI(1,v\vert i^\ast,\,v^\ast)
            f_{i'}(v',t)f_{i^\ast}(v^\ast,t)\,dv'\,dv^\ast\right)dv \\
        &\phantom{=} -\int_\R\varphi(v)\left(\sum_{i^\ast=0}^1\int_\R s^\rmI_{1,i^\ast}(v,v^\ast)f_{i^\ast}(v^\ast,t)\,dv^\ast\right)f_1(v,t)\,dv \\ 
        &= \int_\R\varphi(v)p(v)\rho_0(t)f_1(v,t)\,dv.
    \end{aligned}
    \label{eq:rhsI11}
\end{align}

Next, substituting~\eqref{eq:test_fun} with $l=1$ along with~\eqref{def:sII} and~\eqref{def:mathcalTII} into \tcircle{iii} and \tcircle{iv} on the RHS of~\eqref{eq:generic_form} yields 
\begin{align}
    \begin{aligned}[b]
        \tcircle{iii}+\tcircle{iv} &= \int_\R\varphi(v)\left(\sum_{i''=0}^1\int_\R s^{\rmI\rmI}_{i''}(v'')\T^{\rmI\rmI}(1,v)
            f_{i''}(v'',t)\,dv''\right)dv \\
        &\phantom{=} -\int_\R\varphi(v)s^{\rmI\rmI}_1(v)f_1(v,t)\,dv \\
        &= -\int_\R\varphi(v)d(v)f_1(v,t)\,dv.
    \end{aligned}
    \label{eq:rhsI21}
\end{align}

Finally, substituting~\eqref{eq:test_fun} with $l=1$ along with~\eqref{def:sIII} and~\eqref{def:mathcalTIII} into \tcircle{v} and \tcircle{vi} on the RHS of~\eqref{eq:generic_form} yields 
\begin{align}
    \begin{aligned}[b]
        \tcircle{v}+\tcircle{vi} &= \int_\R\varphi(v)\left(\int_\R s^{\rmI\rmI\rmI}_1\T^{\rmI\rmI\rmI}(v\vert w)f_1(w,t)\,dw\right)dv
            -\int_\R\varphi(v)s^{\rmI\rmI\rmI}_1f_1(v,t)\,dv \\
        &= \mu\int_\R\varphi(v)\left(\int_\R M(v\vert w)f_1(w,t)\,dw\right)dv-\mu\int_\R\varphi(v)f_1(v,t)\,dv.
    \end{aligned}
    \label{eq:rhsI31}
\end{align}

Combining~\eqref{eq:lhsterm1}-\eqref{eq:rhsI31} and rearranging terms we obtain, invoking again the arbitrariness of $\varphi$, the following evolution equation for $f_1$:
\begin{equation}
    \partial_tf_1(v,t)=\bigl(p(v)\rho_0(t)-d(v)\bigr)f_1(v,t)+\mu\left(\int_\R M(v\vert w)f_1(w,t)\,dw-f_1(v,t)\right).
    \label{eq:wfeqf1}
\end{equation}

\paragraph{Mass conservation} Integrating the differential equation~\eqref{eq:wfeqf0} and the IDE~\eqref{eq:wfeqf1} over $\R$, summing together the resulting differential equations, recalling~\eqref{eq:rho} and using the fact that (cf. assumption~\eqref{ass:pdM} on the kernel $M$)
$$ \int_\R M(v\vert w)\,dv=1, \qquad \forall\,w\in\R $$
we find
\begin{align*}
    \dfrac{d}{dt}\bigl(\rho_0(t)+\rho_1(t)\bigr) &= \int_\R d(v'')f_1(v'',t)\,dv''
        -\left(\int_\R p(v^\ast)f_1(v^\ast,t)\,dv^\ast\right)\rho_0(t) \\
    &\phantom{=} +\int_\R\bigl(p(v)\rho_0(t)-d(v)\bigr)f_1(v,t)\,dv=0,
\end{align*}
which implies $\rho_0(t)+\rho_1(t)=\rho_0(0)+\rho_1(0)$ for all $t>0$. Hence, choosing, consistently with~\eqref{eq:globalm}, initial data such that $\rho_0(0)+\rho_1(0)=1$, we have $\rho_0(t)=1-\rho_1(t)$ for all $t>0$. Substituting this into the IDE~\eqref{eq:wfeqf1} we obtain
\begin{equation}
    \partial_tf_1(v,t)=\bigl(p(v)(1-\rho_1(t))-d(v)\bigr)f_1(v,t)+\mu\left(\int_\R M(v\vert w)f_1(w,t)\,dw-f_1(v,t)\right).
    \label{eq:wfeqf1b}
\end{equation}

\paragraph{Mesoscopic model} In the case where $p$ is constant, say $p(v)=p_0\geq 0$, which is the case we shall be focussing on in the remainder of the paper, choosing, without loss of generality, $p_0=1$, introducing the notation below for the \emph{net proliferation rate} (i.e. the difference between the rate of proliferation and the rate of death) 
\begin{equation}
    r(v):=1-d(v),
\label{def:r}
\end{equation}
and renaming $f_1$ to $f$ and $\rho_1$ to $\rho$, from the IDE~\eqref{eq:wfeqf1b} we obtain the following evolution equation for the distribution of population members over the phenotype domain:
\begin{equation}
\begin{cases}
    \partial_tf=\bigl(r(v)-\rho\bigr)f+\mu\left(\displaystyle{\int_\R}M(v\vert w)f(w,t)\,dw-f\right), \\[3mm]
    \rho(t):=\displaystyle{\int_\R}f(v,t)\,dv.
    \label{eq:IDE}
\end{cases}
\end{equation}

Notice that, owing to~\eqref{ass:pdM}, the definition~\eqref{def:r} implies
\begin{subequations}
\begin{equation}
    r\in C(\R), \qquad r:\R\to\R, \qquad r(v)\leq 1\quad \forall\,v\in\R.
\end{equation}
In particular, without loss of generality, throughout the paper we shall assume
\begin{equation}
    \sup_{v\in\R}{r(v)}=1.
\end{equation}
\label{ass:r_with_sup}
\end{subequations}

\section{Quasi-invariant phenotype change regime}
\label{sect:quasi-invariant}
In this section, we consider the regime of \textit{small} but \textit{frequent} phenotype changes, which in the jargon of the kinetic theory of multi-agents systems, cf.~\cite{pareschi2013BOOK}, is called a \textit{quasi-invariant regime}. In this regime, the kernel $M$ is scaled by means of a scaling parameter $0<\varepsilon\ll 1$ in such a way that:
\begin{enumerate}[label=(\roman*)]
\item on the one hand, the phenotype acquired, on average, after a phenotype change, $v$, is a small perturbation of the original phenotype, $w$; \label{item:q.i.-mean}
\item on the other hand, the variability of the newly acquired phenotype, $v$, is small. \label{item:q.i.-variance}
\end{enumerate}

Building upon the scaling for IDE model of evolutionary dynamics in phenotype-structured populations considered in~\cite{barles2009concentration,diekmann2005dynamics,lorz2013populational}, this can be obtained by scaling $M$ as
\begin{equation}
    M(v\vert w) \to M_\varepsilon(v\vert w):=\frac{1}{\varepsilon}\cM\left(\frac{v-w}{\varepsilon};\,\alpha\varepsilon,\beta\right),
    \label{eq:M.scaling}
\end{equation}
where $\alpha\in\R$, $\beta\in\R_+$ are given model parameters and $\cM=\cM(z;\,\xi,\varsigma^2)$ is a two-parameter probability distribution w.r.t. $z\in\R$ with mean $\xi\in\R$ and variance $\varsigma^2\in\R_+$, which we further assume to have bounded third order moment. In detail:
\begin{enumerate}[label=(H\arabic*)]
\item $\displaystyle{\int_\R}\cM(z;\,\xi,\varsigma^2)\,dz=1$ for every $(\xi,\,\varsigma^2)\in\R\times\R_+$; \label{ass:M.mass}
\item $\displaystyle{\int_\R}z\cM(z;\,\xi,\varsigma^2)\,dz=\xi$ for every $\varsigma^2\in\R_+$; \label{ass:M.mean}
\item $\displaystyle{\int_\R}(z-\xi)^2\cM(z;\,\xi,\varsigma^2)\,dz=\varsigma^2$ for every $\xi\in\R$; \label{ass:M.var}
\item $\displaystyle{\int_\R}\abs{z}^3\cM(z;\,\xi,\varsigma^2)\,dz<+\infty$ for every $(\xi,\,\varsigma^2)\in\R\times\R_+$. \label{ass:M.3rd_mom}
\end{enumerate}
Combining \ref{ass:M.mean} and~\ref{ass:M.var} we also deduce
\begin{equation}
    \int_\R z^2\cM(z;\,\xi,\varsigma^2)\,dz=\varsigma^2+\xi^2.
    \label{eq:M.energy}
\end{equation}
On the whole, from~\eqref{eq:M.scaling} and~\ref{ass:M.mass}--\ref{ass:M.var} it is not difficult to verify that, for every $\varepsilon>0$, the following properties hold
\begin{equation}
    \int_\R M_\varepsilon(v\vert w)\,dv=1, \quad \int_\R vM_\varepsilon(v\vert w)\,dv=w+\alpha\varepsilon^2, \quad
        \int_\R\bigl(v-w-\alpha\varepsilon^2\bigr)^2M_\varepsilon(v\vert w)\,dv=\beta\varepsilon^2,
    \label{eq:moments_Meps}
\end{equation}
which confirm that $M_\varepsilon$ describes a regime of small phenotype changes as specified by~\ref{item:q.i.-mean},~\ref{item:q.i.-variance} above. In a moment, we shall examine also the significance of assumption~\ref{ass:M.3rd_mom}.

Scaling $M$ to $M_\varepsilon$ in~\eqref{eq:IDE}, we obtain
\begin{equation}
    \partial_tf_\varepsilon=\bigl(r(v)-\rho_\varepsilon\bigr)f_\varepsilon
        +\mu\left(\int_\R M_\varepsilon(v\vert w)f_\varepsilon(w,t)\,dw-f_\varepsilon\right),
    \label{eq:IDE.scaled}
\end{equation}
where we have denoted by $f_\varepsilon$ the phenotype distribution function parametrised by $\varepsilon$ in consequence of the scaling~\eqref{eq:M.scaling} and by $\rho_\varepsilon(t)=\int_\R f_\varepsilon(v,t)\,dv$ the corresponding density. In weak form, this reads as
\begin{align}
    \resizebox{.93\textwidth}{!}{$\displaystyle
    \begin{aligned}[b]
        \frac{d}{dt}\int_\R\varphi(v)f_\varepsilon(v,t)\,dv &= \int_\R\varphi(v)\bigl(r(v)-\rho_\varepsilon\bigr)f_\varepsilon(v,t)\,dv \\
        &\phantom{=} +\mu\left(\int_\R\int_\R\varphi(v)M_\varepsilon(v\vert w)f_\varepsilon(w,t)\,dw\,dv-\int_\R\varphi(v)f_\varepsilon(v,t)\,dv\right) \\
        &= \int_\R\varphi(v)\bigl(r(v)-\rho_\varepsilon\bigr)f_\varepsilon(v,t)\,dv \\
        &\phantom{=} +\mu\left\{\frac{1}{\varepsilon}\int_\R\left[\int_\R\varphi(v)
            \cM\left(\frac{v-w}{\varepsilon};\,\alpha\varepsilon,\beta\right)dv\right]f_\varepsilon(w,t)\,dw
                -\int_\R\varphi(v)f_\varepsilon(v,t)\,dv\right\},
    \end{aligned}
    $}
    \label{eq:IDE.weak_scaled}
\end{align}
where $\varphi=\varphi(v)$ is a generic observable quantity (test function).

We shall consider non-negative solutions to~\eqref{eq:IDE.scaled},~\eqref{eq:IDE.weak_scaled}, consistently with the interpretation of $f_\varepsilon$ as the statistical distribution of the phenotype $v$ over time. This property is ensured by the following result:
\begin{proposition}[Non-negativity of $f_\varepsilon$] \label{prop:non-neg.feps}
Under~\eqref{ass:r_with_sup}, if $f_\varepsilon(v,0)\geq 0$ for a.e. $v\in\R$ then $f_\varepsilon(v,t)\geq 0$ for a.e. $v\in\R$ and every $t>0$.
\end{proposition}
\begin{proof}
Let $f^\pm_\varepsilon(v,t):=\max\{0,\,\pm f_\varepsilon(v,t)\}\geq 0$ be the positive and negative part, respectively, of $f_\varepsilon$. Writing $f_\varepsilon=f^+_\varepsilon-f^-_\varepsilon$ and noting that $f^+_\varepsilon$, $f^-_\varepsilon$ have essentially disjoint supports, if we multiply~\eqref{eq:IDE.scaled} by $-2f^-_\varepsilon$ we find
$$ \partial_t(f^-_\varepsilon)^2=2\bigl(r(v)-\rho_\varepsilon\bigr)(f^-_\varepsilon)^2
    -2\mu\left(f^-_\varepsilon\int_\R M_\varepsilon(v\vert w)f_\varepsilon(w,t)\,dw+(f^-_\varepsilon)^2\right), $$
from which, integrating over $v\in\R$,
\begin{align*}
    \frac{d}{dt}\int_\R(f^-_\varepsilon(v,t))^2\,dv &= 2\int_\R r(v)(f^-_\varepsilon(v,t))^2\,dv
        -2(\rho_\varepsilon+\mu)\int_\R(f^-_\varepsilon(v,t))^2\,dv \\
    &\phantom{=} -2\mu\int_\R\int_\R M_\varepsilon(v\vert w)f^+_\varepsilon(w,t)f^-_\varepsilon(v,t)\,dw\,dv \\
    &\phantom{=} +2\mu\int_\R\int_\R M_\varepsilon(v\vert w)f^-_\varepsilon(w,t)f^-_\varepsilon(v,t)\,dw\,dv \\
    &\leq 2(1-\rho_\varepsilon-\mu)\int_\R(f^-_\varepsilon(v,t))^2\,dv
        +2\mu\int_\R\int_\R M_\varepsilon(v\vert w)f^-_\varepsilon(w,t)f^-_\varepsilon(v,t)\,dw\,dv.
\end{align*}
Applying repeatedly the Cauchy--Schwarz inequality, first with respect to the probability measure $M_\varepsilon(v\vert w)dv$ and then with respect to the Lebesgue measure $dw$, we obtain
\begin{align*}
    \resizebox{\textwidth}{!}{$\displaystyle
    \begin{aligned}
    \int_\R\int_\R M_\varepsilon(v\vert w)f^-_\varepsilon(w,t)f^-_\varepsilon(v,t)\,dw\,dv &=
        \int_\R\left(\int_\R M_\varepsilon(v\vert w)f^-_\varepsilon(v,t)\,dv\right)f^-_\varepsilon(w,t)\,dw \\
    &\leq \int_\R\left(\int_\R M_\varepsilon(v\vert w)(f^-_\varepsilon(v,t))^2\,dv\right)^{1/2}f^-_\varepsilon(w,t)\,dw \\
    &\leq \left(\int_\R\int_\R M_\varepsilon(v\vert w)(f^-_\varepsilon(v,t))^2\,dv\,dw\right)^{1/2}
        \left(\int_\R (f^-_\varepsilon(w,t))^2\,dw\right)^{1/2} \\
    &= \int_\R(f^-_\varepsilon(v,t))^2\,dv,
    \end{aligned}
    $}
\end{align*}
and thus, finally,
$$ \frac{d}{dt}\int_\R(f^-_\varepsilon(v,t))^2\,dv\leq 2(1-\rho_\varepsilon)\int_\R(f^-_\varepsilon(v,t))^2\,dv. $$
Since $\int_\R(f^-_\varepsilon(v,0))^2\,dv=0$ because $f_\varepsilon(v,0)\geq 0$ a.e. by assumption, the above differential inequality implies
$$ \int_\R(f^-_\varepsilon(v,t))^2\,dv=0, \qquad \forall\,t>0. $$
Consequently, $f^-_\varepsilon(v,t)=0$ for a.e. $v\in\R$ and every $t>0$, which concludes the proof.
\end{proof}

The non-negativity of $f_\varepsilon$ entails:
\begin{proposition}[Non-negativity and boundedness of $\rho_\varepsilon$] \label{prop:bound.rho_eps}
Under~\eqref{ass:r_with_sup}, if $\rho_\varepsilon(0)=\rho^0\leq 1$, where $\rho^0>0$ is taken independent of $\varepsilon$, then $0\leq\rho_\varepsilon(t)\leq 1$ for every $\varepsilon>0$ and every $t>0$.
\end{proposition}
\begin{proof}
With $\varphi(v)=1$ in~\eqref{eq:IDE.weak_scaled} we find
$$ \frac{d\rho_\varepsilon}{dt}=\int_\R r(v)f_\varepsilon(v,t)\,dv-\rho_\varepsilon^2. $$
\begin{enumerate}[label=(\roman*)]
\item To prove the non-negativity of $\rho_\varepsilon$, we multiply this equation by $-2\rho_\varepsilon^-$ to find
$$ \frac{d}{dt}(\rho_\varepsilon^-)^2=-2\rho_\varepsilon^-\int_\R r(v)f_\varepsilon(v,t)\,dv+2(\rho_\varepsilon^-)^3 
    \leq 2(\rho_\varepsilon^-)^3, $$
where we took advantage of the non-negativity of $f_\varepsilon$. Letting $u:=(\rho_\varepsilon^-)^2$, this is equivalent to $\frac{du}{dt}\leq 2u^{3/2}$, which, since $u(0)=(\rho_\varepsilon^-)^2(0)=0$ and $\frac{3}{2}>1$, implies $u(t)\leq 0$ for all $t>0$. Thus, $(\rho_\varepsilon^-)^2(t)=0$ for every $t>0$ and the thesis follows.

\item To prove the boundedness from above of $\rho_\varepsilon$ we observe that
$$ \frac{d\rho_\varepsilon}{dt}\leq\rho_\varepsilon-\rho_\varepsilon^2, $$
due again to the non-negativity of $f_\varepsilon$. Integrating this differential inequality gives
$$ \rho_\varepsilon(t)\leq\frac{\rho^0}{(1-\rho^0)e^{-t}+\rho^0}, $$
from which the thesis follows as $1-\rho^0\geq 0$ by assumption. \qedhere
\end{enumerate}
\end{proof}


Performing now the change of variable $z:=\frac{v-w}{\varepsilon}$ in the $v$-integral containing $\cM$ on the right-hand side of~\eqref{eq:IDE.weak_scaled} gives
\begin{align}
    \resizebox{.93\textwidth}{!}{$\displaystyle
    \begin{aligned}[b]
        \frac{d}{dt}\int_\R\varphi(v)f_\varepsilon(v,t)\,dv &= \int_\R\varphi(v)\bigl(r(v)-\rho_\varepsilon\bigr)f_\varepsilon(v,t)\,dv \\
        &\phantom{=} +\mu\left\{\int_\R\left[\int_\R\varphi(w+\varepsilon z)\cM(z;\,\alpha\varepsilon,\beta)\,dz\right]f_\varepsilon(w,t)\,dw-
            \int_\R\varphi(v)f_\varepsilon(v,t)\,dv\right\}.
    \end{aligned}
    $}
    \label{eq:IDE.weak_change-of-variable}
\end{align}
For a sufficiently smooth test function, say $\varphi\in C^{2,1}(\R)$ with bounded second derivative, since $\varepsilon$ is small we can Taylor-expand $\varphi(w+\varepsilon z)$ up to the second order with centre in $w$ and Lagrange remainder:
$$ \varphi(w+\varepsilon z)=\varphi(w)+\varphi'(w)\varepsilon z+\frac{1}{2}\varphi''(w)\varepsilon^2z^2+
    \frac{1}{2}\bigl(\varphi''(\tilde{v})-\varphi''(w)\bigr)\varepsilon^2z^2, $$
where $\tilde{v}=(1-\theta)w+\theta(w+\varepsilon z)=w+\theta\varepsilon z$ for some $\theta\in [0,\,1]$ \rev{is the point, lying between the points $w$ and $w+\varepsilon z$, at which the Lagrange remainder is evaluated. More precisely, we express it as a convex combination of these two points through a parameter $\theta$, which depends, among other things, on $w$, $w+\varepsilon z$, and the function $\varphi$.} Substituting into~\eqref{eq:IDE.weak_change-of-variable} and recalling~\ref{ass:M.mass},~\ref{ass:M.mean} and~\eqref{eq:M.energy} yields
\begin{align}
    \begin{aligned}[b]
        \frac{d}{dt}\int_\R\varphi(v)f_\varepsilon(v,t)\,dv &= \int_\R\varphi(v)\bigl(r(v)-\rho_\varepsilon\bigr)f_\varepsilon(v,t)\,dv \\
        &\phantom{=} +\mu\varepsilon^2\left(\alpha\int_\R\varphi'(w)f_\varepsilon(w,t)\,dw
            +\frac{1}{2}(\beta+\alpha^2\varepsilon^2)\int_\R\varphi''(w)f_\varepsilon(w,t)\,dw\right. \\
        &\phantom{=} \left.+\frac{1}{2}\int_\R\int_\R\bigl(\varphi''(\tilde{v})-
            \varphi''(w)\bigr)z^2\cM(z;\,\alpha\varepsilon,\beta)f_\varepsilon(w,t)\,dz\,dw\right),
    \end{aligned}
    \label{eq:IDE.weak_with-remainder}
\end{align}
which suggests to fix
\begin{equation}
    \mu=\frac{1}{\varepsilon^2},
    \label{eq:mu.scaling}
\end{equation}
so as to observe the effect of phenotype changes despite the fact that they are small. Notice that this is in line with the idea that, in the quasi-invariant regime, phenotype changes are frequent. With this scaling of $\mu$, we rewrite~\eqref{eq:IDE.weak_with-remainder} as
\begin{align}
    \begin{aligned}[b]
        \frac{d}{dt}\int_\R\varphi(v)f_\varepsilon(v,t)\,dv &= \int_\R\varphi(v)\bigl(r(v)-\rho_\varepsilon\bigr)f_\varepsilon(v,t)\,dv \\
        &\phantom{=} +\alpha\int_\R\varphi'(v)f_\varepsilon(v,t)\,dv
            +\frac{\beta}{2}\int_\R\varphi''(v)f_\varepsilon(v,t)\,dv+R_\varepsilon(f_\varepsilon,\varphi),
    \end{aligned}
    \label{eq:IDE.weak_Reps}
\end{align}
where we have set
$$ R_\varepsilon(f_\varepsilon,\varphi)(t):=\alpha^2\varepsilon^2\int_\R\varphi''(v)f_\varepsilon(v,t)\,dv
    +\frac{1}{2}\int_\R\int_\R\bigl(\varphi''(\tilde{v})-\varphi''(w)\bigr)z^2\cM(z;\,\alpha\varepsilon,\beta)f_\varepsilon(w,t)\,dz\,dw. $$
Owing to Proposition~\ref{prop:bound.rho_eps} and to the assumed smoothness of $\varphi$, we observe that
\begin{align*}
    \abs{R_\varepsilon(f_\varepsilon,\varphi)(t)} &\leq \alpha^2\varepsilon^2\norm{\varphi''}{\infty}
        +\frac{\beta}{2}\operatorname{Lip}{(\varphi'')}\int_\R\int_\R\abs{\tilde{v}-w}z^2\cM(z;\,\alpha\varepsilon,\beta)f_\varepsilon(w,t)\,dz\,dw \\
    &\leq \alpha^2\varepsilon^2\norm{\varphi''}{\infty}
        +\frac{\beta\varepsilon}{2}\operatorname{Lip}{(\varphi'')}\int_\R\abs{z}^3\cM(z;\,\alpha\varepsilon,\beta)\,dz,
\end{align*}
where $\operatorname{Lip}{(\varphi'')}>0$ is the Lipschitz constant of $\varphi''$. In particular, we have used $\abs{\tilde{v}-w}=\theta\varepsilon\abs{z}\leq\varepsilon\abs{z}$ because $0\leq\theta\leq 1$.

To estimate the last integral term on the right-hand side, let $Z_\varepsilon$ be a random variable with law $\cM(z;\,\alpha\varepsilon,\beta)$, so that
$$ \int_\R\abs{z}^3\cM(z;\,\alpha\varepsilon,\beta)\,dz=\ave{\abs{Z_\varepsilon}^3}, $$
where, like in Section~\ref{sect:mesoscopic}, $\ave{\cdot}$ denotes expectation. Then $\ave{Z_\varepsilon}=\alpha\varepsilon$, $\operatorname{Var}{(Z_\varepsilon)}=\beta$ and $Z_\varepsilon$ can be standardised as $Z_\varepsilon=\sqrt{\beta}\tilde{Z}+\alpha\varepsilon$, where $\tilde{Z}$ is a random variable with $\ave{\tilde{Z}}=0$, $\ave{\tilde{Z}^2}=1$, and $\ave{\abs{\tilde{Z}}^3}<+\infty$ owing to~\ref{ass:M.3rd_mom}. Then,
\begin{align}
    \begin{aligned}[b]
        \ave{\abs{Z_\varepsilon}^3}=\ave{\abs{\sqrt{\beta}\tilde{Z}+\alpha\varepsilon}^3} &\leq \beta^{3/2}\ave{\abs{\tilde{Z}}^3}
            +3\abs{\alpha}\beta\varepsilon+3\alpha^2\sqrt{\beta}\varepsilon^2\ave{\abs{\tilde{Z}}}+\abs{\alpha}^3\varepsilon^3 \\
        &\leq \beta^{3/2}\ave{\abs{\tilde{Z}}^3}+1
    \end{aligned}
    \label{eq:aveZ3}
\end{align}
for $\varepsilon$ small enough, since $\ave{\abs{\tilde{Z}}}\leq\ave{\tilde{Z}^2}^{1/2}=1$ by Jensen's inequality.

On the whole, we bound $\abs{R_\varepsilon(f_\varepsilon,\varphi)(t)}$ in such a way that $\abs{R_\varepsilon(f_\varepsilon,\varphi)(t)}\to 0$ as $\varepsilon\to 0^+$. Consequently, we see that, for $\varepsilon$ small, equation~\eqref{eq:IDE.weak_Reps} solved by $f_\varepsilon$ approaches the following equation solved by a new distribution function $g$
\begin{align}
    \begin{aligned}[b]
        \frac{d}{dt}\int_\R\varphi(v)g(v,t)\,dv &= \int_\R\varphi(v)\bigl(r(v)-\varrho\bigr)g(v,t)\,dv \\
        &\phantom{=} +\alpha\int_\R\varphi'(v)g(v,t)\,dv+\frac{\beta}{2}\int_\R\varphi''(v)g(v,t)\,dv
    \end{aligned}
    \label{eq:PDE.weak}
\end{align}
with $\varrho(t):=\int_\R g(v,t)\,dv$.

Equation~\eqref{eq:PDE.weak} can be fruitfully written in strong form by integrating by parts the last two terms on the right-hand side. This yields
\begin{equation}
    \partial_tg=\bigl(r(v)-\varrho\bigr)g-\alpha\partial_vg+\frac{\beta}{2}\partial^2_vg,
    \label{eq:PDE.strong}
\end{equation}
provided that suitable conditions are imposed for $\abs{v}\to\infty$. Such conditions involve, in general, the test function $\varphi$ but, as we shall see, only test functions in the form of powers of $v$ will be relevant for our next purposes. In more detail, for $\varphi$ such that $\abs{\varphi(v)}\sim\abs{v}^p$ when $\abs{v}\to\infty$ with $0\leq p\leq 2$,~\eqref{eq:PDE.strong} is obtained from~\eqref{eq:PDE.weak} by prescribing the ``boundary'' conditions
\begin{equation}
    g(v,t),\,\partial_vg(v,\,t)\xrightarrow{\abs{v}\to\infty}{0} \quad \forall\,t>0,
        \quad \text{infinitesimals of order} >2.
    \label{eq:bc_g}
\end{equation}

Notice that~\eqref{eq:PDE.strong} is a Fokker--Planck-type equation with non-local reaction term. Also in this case, we shall consider non-negative solutions to~\eqref{eq:PDE.weak},~\eqref{eq:PDE.strong}, which are provided by
\begin{proposition}[Non-negativity of $g$] \label{prop:non-neg.g}
Under~\eqref{ass:r_with_sup} and~\eqref{eq:bc_g}, if $g(v,0)\geq 0$ for a.e. $v\in\R$ then $g(v,t)\geq 0$ for a.e. $v\in\R$ and every $t>0$.
\end{proposition}
\begin{proof}
We multiply~\eqref{eq:PDE.strong} by $-2g^-$, where $g^-$ denotes the negative part of $g$, to obtain
$$ \partial_t(g^-)^2=2\bigl(r(v)-\varrho\bigr)(g^-)^2-\alpha\partial_v(g^-)^2-\beta g^-\partial^2_vg. $$
Next, we integrate over $v\in\R$ and use integration by parts to find
\begin{align*}
    \frac{d}{dt}\int_\R(g^-(v,t))^2\,dv &= 2\int_\R r(v)(g^-(v,t))^2\,dv-2\varrho\int_\R(g^-(v,t))^2\,dv \\
    &\phantom{=} -\alpha\Bigl((g^-(v,t))^2\Bigr\vert_{v=-\infty}^{v=+\infty} \\
    &\phantom{=} -\beta\left(-\int_\R\partial_vg(v,t)\partial_vg^-(v,t)\,dv+\Bigl(g^-(v,t)\partial_vg(v,t)\Bigr\vert_{v=-\infty}^{v=+\infty}\right).
\end{align*}
Owing to~\eqref{eq:bc_g}, the boundary terms vanish. Moreover, writing $g=g^+-g^-$ and using that $g^+$, $g^-$ have essentially disjoint supports we deduce $\int_\R\partial_vg(v,t)\partial_vg^-(v,t)\,dv=-\int_\R(\partial_vg^-(v,t))^2\,dv\leq 0$, and thus, finally,
$$ \frac{d}{dt}\int_\R(g^-(v,t))^2\,dv\leq 2(1-\varrho)\int_\R(g^-(v,t))^2\,dv, $$
which, since $\int_\R(g^-(v,0))^2\,dv=0$ because $g(v,0)\geq 0$ a.e. by assumption, implies
$$ \int_\R(g^-(v,t))^2\,dv=0, \qquad \forall\,t>0. $$
Therefore, $g^-(v,t)=0$ for a.e. $v\in\R$ and every $t>0$, which concludes the proof.
\end{proof}

\subsection{Convergence of~\texorpdfstring{$\boldsymbol{f_\varepsilon}$}{} to~\texorpdfstring{$\boldsymbol{g}$}{} as~\texorpdfstring{$\boldsymbol{\varepsilon\to 0^+}$}{}}
In this section, we prove the convergence of the solution $f_\varepsilon$ of the IDE model~\eqref{eq:IDE.scaled} to the solution $g$ of the non-local PDE model~\eqref{eq:PDE.strong}-\eqref{eq:bc_g} in the quasi-invariant limit $\varepsilon\to 0^+$. This will provide a mathematical justification of the fact that the non-local PDE model, often used in the literature as the starting point of further investigations, is the exact counterpart of the agent-based model in the quasi-invariant regime.

\subsubsection{\textit{A priori} estimates}
To begin with, we establish some \textit{a priori} estimates which will be useful in the sequel.

\begin{proposition}[Non-negativity and boundedness of $\varrho$] \label{prop:bound.varrho}
Under the same assumptions as in Proposition~\ref{prop:bound.rho_eps}, if $\varrho(0)=\varrho^0\leq 1$ then $0\leq\varrho(t)\leq 1$ for every $t>0$.
\end{proposition}
\begin{proof}
From~\eqref{eq:PDE.weak}, we obtain the evolution equation of $\varrho$ by taking $\varphi(v)=1$:
$$ \frac{d\varrho}{dt}=\int_\R r(v)g(v,t)\,dv-\varrho^2. $$
The thesis follows then by arguing like in the proof of Proposition~\ref{prop:bound.rho_eps}, since $\varrho$ fulfills the same equation as $\rho_\varepsilon$ with $f_\varepsilon$ replaced by the non-negative distribution function $g$.
\end{proof}

From now on, we shall invariably make the following assumptions on the function $r$ and the initial phenotype distribution $f^0=f^0(v)$:
\begin{align}
    \begin{aligned}[c]
        & r,\,f^0\in L^2(\R)\cap W^{3,\infty}(\R), \\
        & f^0,\,\partial_vf^0\to 0 \quad \text{as} \quad \abs{v}\to\infty, \qquad
            f^0(v)\geq 0 \quad \text{a.e.}
    \end{aligned}
    \label{eq:r,f0}
\end{align}
The assumption on the trend of $f^0$, $\partial_vf^0$ for $\abs{v}\to\infty$ is motivated by the similar request on $g$, cf.~\eqref{eq:bc_g}. More specifically, we prescribe $f^0$ as the common initial condition to both~\eqref{eq:IDE.scaled} for every $\varepsilon>0$ and~\eqref{eq:PDE.strong}:
\begin{equation}
    f_\varepsilon(\cdot,0)=g(\cdot,0)=f^0, \qquad \forall\,\varepsilon>0.
    \label{eq:init_cond}
\end{equation}

Notice that the results of Propositions~\ref{prop:bound.rho_eps},~\ref{prop:bound.varrho} hold using, in particular, $\norm{r}{\infty}$ as an upper bound on $r$, since $\norm{r}{\infty}\geq 1$ due to~\eqref{ass:r_with_sup}.

\begin{remark} \label{rem:constant_r}
The theory we develop in this paper does not encompass the case of constant $r$, cf. the assumption $r\in L^2(\R)$ in~\eqref{eq:r,f0}. The reason is that such a case, though simpler, is structurally quite different from the general case of non-constant $r$. Specifically, it gives rise to self-consistent equations for the densities $\rho_\varepsilon$, $\varrho$ with also $\rho_\varepsilon(t)=\varrho(t)$ for every $\varepsilon,\,t>0$. This impacts on the estimates needed to achieve the quasi-invariant limit in a way which cannot be dealt with simply as a particular instance of the case of non-constant $r$.

While throughout the paper we focus on the biologically more significant case of a net proliferation rate, $r$,  depending on the phenotype of the individuals, in Appendix~\ref{app:const_r} we provide for completeness technical details about the case of constant $r$.
\end{remark}

\begin{proposition}[$L^2$ estimate on $f_\varepsilon$] \label{prop:L2_feps}
Under~\eqref{eq:r,f0}, $f_\varepsilon(\cdot,t)\in L^2(\R)$ for every $\varepsilon>0$ and every $t\in (0,\,T]$, where $T>0$ is arbitrary.
\end{proposition}
\begin{proof}
We multiply~\eqref{eq:IDE.scaled} by $2f_\varepsilon$ and integrate both sides w.r.t. $v\in\R$ to find
\begin{align*}
    \frac{d}{dt}\norm{f_\varepsilon(t)}{L^2}^2 &= 2\int_\R\bigl(r(v)-\rho_\varepsilon\bigr)f_\varepsilon^2(v,t)\,dv \\
    &\phantom{=} +\frac{2}{\varepsilon^2}\int_\R\int_\R M_\varepsilon(v\vert w)f_\varepsilon(w,t)f_\varepsilon(v,t)\,dv\,dw
        -\frac{2}{\varepsilon^2}\norm{f_\varepsilon}{L^2}^2 \\
    &\leq 2\left(\norm{r}{\infty}-\frac{1}{\varepsilon^2}\right)\norm{f_\varepsilon(t)}{L^2}^2
        +\frac{2}{\varepsilon^2}\int_\R\left(\int_\R M_\varepsilon(v\vert w)f_\varepsilon(v,t)\,dv\right)f_\varepsilon(w,t)\,dw.
\intertext{Since $\int_\R M_\varepsilon(v\vert w)f_\varepsilon(v,t)\,dv\leq\left(\int_\R M_\varepsilon(v\vert w)f_\varepsilon^2(v,t)\,dv\right)^{1/2}$ by Cauchy--Schwarz inequality w.r.t. the integration probability measure $M_\varepsilon(v\vert w)dv$, then}
    &\leq 2\left(\norm{r}{\infty}-\frac{1}{\varepsilon^2}\right)\norm{f_\varepsilon(t)}{L^2}^2
        +\frac{2}{\varepsilon^2}\int_\R\left(\int_\R M_\varepsilon(v\vert w)f_\varepsilon^2(v,t)\,dv\right)^{1/2}f_\varepsilon(w,t)\,dw.
\intertext{A further application of Cauchy--Schwarz inequality to the integral in $dw$ produces}
    &\leq 2\left(\norm{r}{\infty}-\frac{1}{\varepsilon^2}\right)\norm{f_\varepsilon(t)}{L^2}^2
        +\frac{2}{\varepsilon^2}\norm{f_\varepsilon(t)}{L^2}\left(\int_\R\int_\R M_\varepsilon(v\vert w)f_\varepsilon^2(v,t)\,dv\,dw\right)^{1/2},
\intertext{and thus, invoking the first formula in~\eqref{eq:moments_Meps},}
    &= 2\norm{r}{\infty}\norm{f_\varepsilon(t)}{L^2}^2.
\end{align*}

Finally,
$$ \norm{f_\varepsilon(t)}{L^2}\leq\norm{f^0}{L^2}e^{\norm{r}{\infty}t}, $$
which shows that $\norm{f_\varepsilon(t)}{L^2}$ is bounded for every $t\in (0,\,T]$ with arbitrary $T>0$, and the thesis then follows.
\end{proof}

\begin{proposition}[$L^\infty$ estimates of $f_\varepsilon$ and its derivatives] \label{prop:Linf_estimates}
Under~\eqref{eq:r,f0}, $f_\varepsilon(\cdot,t)\in W^{3,\infty}(\R)$ for every $\varepsilon>0$ and every $t\in (0,\,T]$, where $T>0$ is arbitrary.
\end{proposition}
\begin{proof}
\begin{enumerate}[label=(\roman*)]
\item We begin by estimating $\norm{f_\varepsilon(t)}{\infty}$. From~\eqref{eq:IDE.scaled} with the scaling~\eqref{eq:M.scaling},~\eqref{eq:mu.scaling} and the change of variable $z:=\frac{v-w}{\varepsilon}$ we have
\begin{equation}
    \partial_tf_\varepsilon+\frac{1}{\varepsilon^2}f_\varepsilon=\bigl(r(v)-\rho_\varepsilon\bigr)f_\varepsilon+
        \frac{1}{\varepsilon^2}\int_\R\cM(z;\,\alpha\varepsilon,\beta)f_\varepsilon(v-\varepsilon z,t)\,dz,
    \label{eq:feps_for_norminf}
\end{equation}
and then, using the non-negativity of $f_\varepsilon$, $\rho_\varepsilon$ and multiplying both sides by $e^{t/\varepsilon^2}$,
$$ \partial_t\left(e^{t/\varepsilon^2}f_\varepsilon\right)\leq\left(\norm{r}{\infty}+\frac{1}{\varepsilon^2}\right)e^{t/\varepsilon^2}\norm{f_\varepsilon(t)}{\infty}. $$
We integrate now both sides on $[0,\,t]$, $t>0$, to find
$$ e^{t/\varepsilon^2}f_\varepsilon(v,t)\leq f^0(v)
    +\left(\norm{r}{\infty}+\frac{1}{\varepsilon^2}\right)\int_0^te^{s/\varepsilon^2}\norm{f_\varepsilon(s)}{\infty}\,ds $$
for every $v\in\R$. Therefore
$$ e^{t/\varepsilon^2}\norm{f_\varepsilon(t)}{\infty}\leq\norm{f^0}{\infty}
    +\left(\norm{r}{\infty}+\frac{1}{\varepsilon^2}\right)\int_0^te^{s/\varepsilon^2}\norm{f_\varepsilon(s)}{\infty}\,ds $$
and, finally, applying Gr\"{o}nwall's inequality to the function $e^{t/\varepsilon^2}\norm{f_\varepsilon(t)}{\infty}$,
$$ \norm{f_\varepsilon(t)}{\infty}\leq\norm{f^0}{\infty}e^{\norm{r}{\infty}t}, $$
which ensures that $\norm{f_\varepsilon(t)}{\infty}$ is bounded for every $t\in (0,\,T]$ with arbitrary $T>0$.
\item We estimate now $\norm{\partial_vf_\varepsilon(t)}{\infty}$. For this, we differentiate~\eqref{eq:feps_for_norminf} w.r.t. $v$ to find
\begin{equation}
    \partial_t(\partial_vf_\varepsilon)+\frac{1}{\varepsilon^2}\partial_vf_\varepsilon=r'(v)f_\varepsilon
        +\bigl(r(v)-\rho_\varepsilon\bigr)\partial_vf_\varepsilon+
            \frac{1}{\varepsilon^2}\int_\R\cM(z;\,\alpha\varepsilon,\beta)\partial_vf_\varepsilon(v-\varepsilon z,t)\,dz,
    \label{eq:dfeps_for_norminf}
\end{equation}
then we multiply both sides by $e^{t/\varepsilon^2}$ and obtain the estimate
$$ \partial_t\left(e^{t/\varepsilon^2}\partial_vf_\varepsilon\right)\leq\norm{r'}{\infty}e^{t/\varepsilon^2}\norm{f_\varepsilon(t)}{\infty}
    +\left(2\norm{r}{\infty}+\frac{1}{\varepsilon^2}\right)e^{t/\varepsilon^2}\norm{\partial_vf_\varepsilon(t)}{\infty}, $$
from which, integrating both sides on $[0,\,t]$, $t>0$, and arguing like in the previous step,
\begin{align*}
    e^{t/\varepsilon^2}\norm{\partial_vf_\varepsilon(t)}{\infty} &\leq \norm{\partial_vf^0}{\infty}
        +\norm{r'}{\infty}\int_0^te^{s/\varepsilon^2}\norm{f_\varepsilon(s)}{\infty}\,ds \\
    &\phantom{\leq} +\left(2\norm{r}{\infty}+\frac{1}{\varepsilon^2}\right)\int_0^te^{s/\varepsilon^2}\norm{\partial_vf_\varepsilon(s)}{\infty}\,ds.
\end{align*}
Finally, applying Lemma~\ref{lemma:Gronw_Tosc} (see Appendix~\ref{appendixB}) to the function $u(t)=e^{t/\varepsilon^2}\norm{\partial_vf_\varepsilon(t)}{\infty}$ with
\begin{equation*}
    a(t)=\norm{\partial_vf^0}{\infty}+\norm{r'}{\infty}\int_0^te^{s/\varepsilon^2}\norm{f_\varepsilon(s)}{\infty}\,ds, \qquad
        b(t)=2\norm{r}{\infty}+\frac{1}{\varepsilon^2},
\end{equation*}
we obtain
\begin{align*}
    \norm{\partial_vf_\varepsilon(t)}{\infty} &\leq \norm{\partial_vf^0}{\infty}e^{2\norm{r}{\infty}t}+
        \norm{r'}{\infty}\int_0^t\norm{f_\varepsilon(s)}{\infty}e^{2\norm{r}{\infty}(t-s)}\,ds \\
    &\leq \left(\norm{\partial_vf^0}{\infty}+\norm{r'}{\infty}\norm{f^0}{\infty}t\right)e^{2\norm{r}{\infty}t},
\end{align*}
where we have used the estimate of $\norm{f_\varepsilon(t)}{\infty}$ established in the previous step. This shows that $\norm{\partial_vf_\varepsilon(t)}{\infty}$ is bounded for every $t\in (0,\,T]$ for an arbitrary $T>0$.
\item We continue by estimating $\norm{\partial^2_vf_\varepsilon(t)}{\infty}$. Differentiating~\eqref{eq:dfeps_for_norminf} w.r.t. $v$ gives
\begin{align}
    \begin{aligned}[b]
        \partial_t(\partial^2_vf_\varepsilon)+\frac{1}{\varepsilon^2}\partial^2_vf_\varepsilon &= r''(v)f_\varepsilon
            +2r'(v)\partial_vf_\varepsilon+\bigl(r(v)-\rho_\varepsilon\bigr)\partial^2_vf_\varepsilon \\
        &\phantom{=} +\frac{1}{\varepsilon^2}\int_\R\cM(z;\,\alpha\varepsilon,\beta)\partial^2_vf_\varepsilon(v-\varepsilon z,t)\,dz,
    \end{aligned}
    \label{eq:ddfeps_for_norminf}
\end{align}
and then, proceeding analogously to the previous steps, we find
\begin{align*}
    \partial_t\left(e^{t/\varepsilon^2}\partial^2_vf_\varepsilon\right) &\leq \norm{r''}{\infty}e^{t/\varepsilon^2}\norm{f_\varepsilon(t)}{\infty}
        +2\norm{r'}{\infty}e^{t/\varepsilon^2}\norm{\partial_vf_\varepsilon(t)}{\infty} \\
    &\phantom{=} +\left(2\norm{r}{\infty}+\frac{1}{\varepsilon^2}\right)e^{t/\varepsilon^2}\norm{\partial^2_vf_\varepsilon(t)}{\infty}
\end{align*}
and further
\begin{align*}
    e^{t/\varepsilon^2}\norm{\partial^2_vf_\varepsilon(t)}{\infty} &\leq \norm{\partial^2_vf^0}{\infty}
        +\int_0^te^{s/\varepsilon^2}\left(\norm{r''}{\infty}\norm{f_\varepsilon(s)}{\infty}
        +2\norm{r'}{\infty}\norm{\partial_vf_\varepsilon(s)}{\infty}\right)ds \\
    &\phantom{=} +\left(2\norm{r}{\infty}+\frac{1}{\varepsilon^2}\right)
        \int_0^te^{s/\varepsilon^2}\norm{\partial^2_vf_\varepsilon(s)}{\infty}\,ds
\end{align*}
for $t>0$. Applying now Lemma~\ref{lemma:Gronw_Tosc} (see Appendix~\ref{appendixB}) to the function $u(t)=e^{t/\varepsilon^2}\norm{\partial^2_vf_\varepsilon(t)}{\infty}$ with
\begin{align*}
    a(t) &= \norm{\partial^2_vf^0}{\infty}+\int_0^te^{s/\varepsilon^2}\left(\norm{r''}{\infty}\norm{f_\varepsilon(s)}{\infty}
        +2\norm{r'}{\infty}\norm{\partial_vf_\varepsilon(s)}{\infty}\right)ds, \\
    b(t) &= 2\norm{r}{\infty}+\frac{1}{\varepsilon^2}
\end{align*}
yields
\begin{align*}
    \norm{\partial^2_vf_\varepsilon(t)}{\infty} &\leq \norm{\partial^2_vf^0}{\infty}e^{2\norm{r}{\infty}t} \\
    &\phantom{=} +\int_0^t\left(\norm{r''}{\infty}\norm{f_\varepsilon(s)}{\infty}
        +2\norm{r'}{\infty}\norm{\partial_vf_\varepsilon(s)}{\infty}\right)e^{2\norm{r}{\infty}(t-s)}\,ds,
\end{align*}
and thus the boundedness of $\norm{\partial^2_vf_\varepsilon(t)}{\infty}$ for every $t\in (0,\,T]$ and arbitrary $T>0$ follows owing to the boundedness of $\norm{f_\varepsilon(t)}{\infty}$, $\norm{\partial_vf_\varepsilon(t)}{\infty}$ established in the previous steps.
\item We conclude by estimating $\norm{\partial^3_vf_\varepsilon(t)}{\infty}$. For this, we differentiate~\eqref{eq:ddfeps_for_norminf} w.r.t. $v$ obtaining:
\begin{align*}
    \partial_t(\partial^3_vf_\varepsilon)+\frac{1}{\varepsilon^2}\partial^3_vf_\varepsilon &=
        r'''(v)f_\varepsilon+3r''(v)\partial_vf_\varepsilon+3r'(v)\partial^2_vf_\varepsilon+\bigl(r(v)-\rho_\varepsilon\bigr)\partial^3_vf_\varepsilon \\
    &\phantom{=} +\frac{1}{\varepsilon^2}\int_\R\cM(z;\,\alpha\varepsilon,\beta)\partial^3_vf_\varepsilon(v-\varepsilon z,t)\,dz \\
    &\leq \norm{r'''}{\infty}\norm{f_\varepsilon(t)}{\infty}+3\norm{r''}{\infty}\norm{\partial_vf_\varepsilon(t)}{\infty}
        +3\norm{r'}{\infty}\norm{\partial^2_vf_\varepsilon(t)}{\infty} \\
    &\phantom{\leq} +\left(2\norm{r}{\infty}+\frac{1}{\varepsilon^2}\right)\norm{\partial^3_vf_\varepsilon(t)}{\infty},
\end{align*}
and thus, multiplying both sides by $e^{t/\varepsilon^2}$ and integrating on $[0,\,t]$, $t>0$, we find
\begin{align*}
    \resizebox{.94\textwidth}{!}{$\displaystyle
    \begin{aligned}
        e^{t/\varepsilon^2}\norm{\partial^3_vf_\varepsilon(t)}{\infty} &\leq \norm{\partial^3_vf^0}{\infty} \\
        &\phantom{\leq} +\int_0^te^{s/\varepsilon^2}(\norm{r'''}{\infty}\norm{f_\varepsilon(s)}{\infty}+3\norm{r''}{\infty}\norm{\partial_vf_\varepsilon(s)}{\infty}
            +3\norm{r'}{\infty}\norm{\partial^2_vf_\varepsilon(s)}{\infty})\,ds \\
        &\phantom{\leq} +\left(2\norm{r}{\infty}+\frac{1}{\varepsilon^2}\right)\int_0^te^{s/\varepsilon^2}\norm{\partial^3_vf_\varepsilon(s)}{\infty}\,ds.
    \end{aligned}
    $}
\end{align*}
We now apply Lemma~\ref{lemma:Gronw_Tosc} (see Appendix~\ref{appendixB}) to the function $u(t)=e^{t/\varepsilon^2}\norm{\partial^3_vf_\varepsilon(t)}{\infty}$ with
\begin{align*}
    a(t) &= \norm{\partial^3_vf^0}{\infty}
        +\int_0^te^{s/\varepsilon^2}(\norm{r'''}{\infty}\norm{f_\varepsilon(s)}{\infty}+3\norm{r''}{\infty}\norm{\partial_vf_\varepsilon(s)}{\infty}
            +3\norm{r'}{\infty}\norm{\partial^2_vf_\varepsilon(s)}{\infty})\,ds \\
    b(t) &= 2\norm{r}{\infty}+\frac{1}{\varepsilon^2}
\end{align*}
to find
\begin{align*}
    \norm{\partial^3_vf_\varepsilon(t)}{\infty} &\leq \norm{\partial^3_vf^0}{\infty}e^{2\norm{r}{\infty}t} \\
    &\phantom{\leq} +\int_0^t(\norm{r'''}{\infty}\norm{f_\varepsilon(s)}{\infty}+3\norm{r''}{\infty}\norm{\partial_vf_\varepsilon(s)}{\infty} \\
    &\phantom{\leq+\int_0^t}
        +3\norm{r'}{\infty}\norm{\partial^2_vf_\varepsilon(s)}{\infty})e^{2\norm{r}{\infty}(t-s)}\,ds.
\end{align*}
The boundedness of $\norm{\partial^3_vf_\varepsilon(t)}{\infty}$ for every $t\in (0,\,T]$ and arbitrary $T>0$ follows then from the analogous property of $\norm{f_\varepsilon(t)}{\infty}$, $\norm{\partial_vf_\varepsilon(t)}{\infty}$ and $\norm{\partial^2_vf_\varepsilon(t)}{\infty}$ obtained in the previous steps. \qedhere
\end{enumerate}
\end{proof}

\begin{remark}
As a by-product of Proposition~\ref{prop:Linf_estimates}, we notice that $\norm{f_\varepsilon(t)}{\infty}$, $\norm{\partial_vf_\varepsilon(t)}{\infty}$, $\norm{\partial^2_vf_\varepsilon(t)}{\infty}$, and $\norm{\partial^3_vf_\varepsilon(t)}{\infty}$ are all bounded \textit{uniformly} in $\varepsilon$ on every bounded interval $(0,\,T]$, $T>0$.
\end{remark}

Unlike~\eqref{eq:PDE.strong}, the IDE~\eqref{eq:IDE.scaled} does not require the prescription of boundary conditions. Nevertheless, soon a clue to the trend of $f_\varepsilon(\cdot,t)$, $\partial_vf_\varepsilon(\cdot,t)$ for $\abs{v}\to\infty$ will be needed, which is provided by
\begin{proposition} \label{prop:bc_feps}
Under~\eqref{eq:r,f0}, $f_\varepsilon(v,t),\,\partial_vf_\varepsilon(v,t)\to 0$ for $\abs{v}\to\infty$ for every $t>0$ and every $\varepsilon>0$.
\end{proposition}
\begin{proof}
\begin{enumerate}[label=(\roman*)]
\item Let $u_\varepsilon(t):=\lim_{v\to +\infty}f_\varepsilon(v,t)$. Passing to the limit $v\to +\infty$ in~\eqref{eq:feps_for_norminf} we obtain
$$ u_\varepsilon'+\frac{1}{\varepsilon^2}u_\varepsilon\leq 2\norm{r}{\infty}u_\varepsilon+\frac{1}{\varepsilon^2}u_\varepsilon, $$
where we observed that, owing to Proposition~\ref{prop:Linf_estimates}, the mapping $z\mapsto\cM(z;\,\alpha\varepsilon,\beta)f_\varepsilon(v-\varepsilon z,t)$ is bounded for every $v\in\R$ by the integrable mapping $z\mapsto\norm{f_\varepsilon(t)}{\infty}\cM(z;\,\alpha\varepsilon,\beta)$, cf.~\ref{ass:M.mass}, whereby we could pass the limit through the integral by dominated convergence. The inequality above is equivalent to
$$ u_\varepsilon'\leq 2\norm{r}{\infty}u_\varepsilon, $$
which, along with $u_\varepsilon(0)=\lim_{v\to +\infty}f^0(v)=0$ by assumption, implies $u_\varepsilon(t)\leq 0$ for all $t>0$. But $u_\varepsilon(t)\geq 0$, because $f_\varepsilon$ is non-negative; thus, ultimately,  $u_\varepsilon(t)=0$ for all $t>0$.

The very same argument allows one to prove also that $\lim_{v\to -\infty}f_\varepsilon(v,t)=0$ for all $t>0$.

\item Now let $w_\varepsilon(t):=\lim_{v\to +\infty}\abs{\partial_vf_\varepsilon(v,t)}\geq 0$. From~\eqref{eq:dfeps_for_norminf} we deduce preliminarily
$$ \partial_t\abs{\partial_vf_\varepsilon}\leq\norm{r'}{\infty}f_\varepsilon+2\norm{r}{\infty}\abs{\partial_vf_\varepsilon}
    +\frac{1}{\varepsilon^2}\left(\int_\R\cM(z;\,\alpha\varepsilon,\beta)\abs{\partial_vf_\varepsilon(v-\varepsilon z,t)}\,dz+\abs{\partial_vf_\varepsilon}\right), $$
and then, taking the limit $v\to +\infty$ and invoking again the dominated convergence,
$$ w_\varepsilon'\leq 2\left(\norm{r}{\infty}+\frac{1}{\varepsilon^2}\right)w_\varepsilon, $$
where we have used that $f_\varepsilon(\cdot,t)\to 0$ for $v\to +\infty$ as proved in the previous step. Since $w_\varepsilon(0)=\lim_{v\to +\infty}\abs{\partial_vf^0(v)}=0$ by assumption, we conclude $w_\varepsilon(t)=0$ for all $t>0$.

The same argument allows one to \rev{conclude} also that $\lim_{v\to -\infty}\abs{\partial_vf_\varepsilon(v,t)}=0$ for all $t>0$. \qedhere
\end{enumerate}
\end{proof}

We end this section by showing that also the non-local PDE model~\eqref{eq:PDE.strong} preserves the boundedness in time of the $L^2$ norm of its solution:
\begin{proposition}[$L^2$ estimate on $g$] \label{prop:L2_g}
Under~\eqref{eq:r,f0}, $g(\cdot,t)\in L^2(\R)$ for every $t\in (0,\,T]$, where $T>0$ is arbitrary.
\end{proposition}
\begin{proof}
We multiply~\eqref{eq:PDE.strong} by $2g$ and integrate both sides w.r.t. $v\in\R$ to find:
$$ \frac{d}{dt}\norm{g(t)}{L^2}^2=2\int_\R\bigl(r(v)-\varrho)g^2(v,t)\,dv-2\alpha\int_\R g(v,t)\partial_vg(v,t)\,dv
    +\beta\int_\R g(v,t)\partial^2_vg(v,t)\,dv. $$
Observing that $2g\partial_vg=\partial_vg^2$ and integrating by parts the last term on the right-hand side we further obtain:
\begin{align*}
    \frac{d}{dt}\norm{g(t)}{L^2}^2 &= 2\int_\R\bigl(r(v)-\varrho)g^2(v,t)\,dv-\alpha\Bigl(g^2(v,t)\Bigr\vert_{v=-\infty}^{v=+\infty} \\
    &\phantom{=} -\beta\int_\R\left(\partial_vg(v,t)\right)^2dv+\beta\Bigl(g(v,t)\partial_vg(v,t)\Bigr\vert_{v=-\infty}^{v=+\infty}.
\intertext{Finally, the boundary conditions~\eqref{eq:bc_g} together with the non-negativity of $(\partial_vg)^2$ and $\varrho$ yield}
    &\leq 2\norm{r}{\infty}\norm{g(t)}{L^2}^2;
\end{align*}
therefore,
$$ \norm{g(t)}{L^2}\leq\norm{f^0}{L^2}e^{\norm{r}{\infty}t}, $$
and the thesis follows.
\end{proof}

\subsubsection{\texorpdfstring{$\boldsymbol{L^2}$}{} convergence} \label{sect:L2_convergence}
Propositions~\ref{prop:L2_feps},~\ref{prop:L2_g} indicate that, for $L^2$-integrable initial data, $L^2(\R)$ is a space  the solutions $f_\varepsilon(\cdot,t)$, $g(\cdot,t)$ to~\eqref{eq:IDE.scaled},~\eqref{eq:PDE.strong} belong to for $t$ in bounded intervals of the form $(0,\,T]$, $T>0$ arbitrary. It is therefore reasonable to ascertain the convergence of $f_\varepsilon$ to $g$ in that space as $\varepsilon\to 0^+$.

To this purpose, we preliminarily establish:
\begin{proposition} \label{prop:rhoeps-varrho}
Under~\eqref{eq:r,f0},~\eqref{eq:init_cond}, the following estimate holds:
$$ \abs{(\rho_\varepsilon-\varrho)(t)}\leq\norm{r}{L^2}e^{2\norm{r}{\infty}t}\int_0^t\norm{(f_\varepsilon-g)(s)}{L^2}\,ds,
    \qquad t>0. $$
\end{proposition}
\begin{proof}
Subtracting~\eqref{eq:IDE.weak_scaled} and~\eqref{eq:PDE.weak} with $\varphi(v)=1$ gives
$$ \frac{d}{dt}(\rho_\varepsilon-\varrho)=\int_\R r(v)\bigl(f_\varepsilon(v,t)-g(v,t)\bigr)\,dv
    -\bigl(\rho_\varepsilon^2-\varrho^2\bigr). $$
Using the Cauchy--Schwarz inquality in the first term on the right-hand side and writing $\rho_\varepsilon^2-\varrho^2=(\rho_\varepsilon+\varrho)(\rho_\varepsilon-\varrho)$, we estimate
$$ \frac{d}{dt}\abs{\rho_\varepsilon-\varrho}\leq \norm{r}{L^2}\norm{(f_\varepsilon-g)(t)}{L^2}+2\norm{r}{\infty}\abs{\rho_\varepsilon-\varrho}, $$
whence, integrating over $[0,\,t]$, $t>0$, and using that $\rho_\varepsilon(0)=\varrho(0)$ owing to~\eqref{eq:init_cond},
$$ \abs{(\rho_\varepsilon-\varrho)(t)}\leq \norm{r}{L^2}\int_0^t\norm{(f_\varepsilon-g)(s)}{L^2}\,ds
    +2\norm{r}{\infty}\int_0^t\abs{(\rho_\varepsilon-\varrho)(s)}\,ds. $$
Gr\"{o}nwall's inequality yields then the thesis.
\end{proof}

We are now in a position to prove:
\begin{theorem}[Quasi-invariant limit] \label{theo:feps_to_g}
Under~\eqref{eq:r,f0},~\eqref{eq:init_cond}, $f_\varepsilon(\cdot,t)\to g(\cdot,t)$ in $L^2(\R)$ for every $t\in (0,\,T]$ with arbitrary $T>0$ when $\varepsilon\to 0^+$.
\end{theorem}
\begin{proof}
Subtracting~\eqref{eq:IDE.scaled} -- with the scaling~\eqref{eq:M.scaling},~\eqref{eq:mu.scaling} -- and~\eqref{eq:PDE.strong} we find
\begin{align*}
    \partial_t(f_\varepsilon-g) &= \bigl(r(v)-\rho_\varepsilon\bigr)(f_\varepsilon-g)-(\rho_\varepsilon-\varrho)g \\
    &\phantom{=} +\frac{1}{\varepsilon^2}\left(\int_\R\cM(z;\,\alpha\varepsilon,\beta)
        f_\varepsilon(v-\varepsilon z,t)\,dz-f_\varepsilon\right)+\alpha\partial_vg-\frac{\beta}{2}\partial^2_vg,
\end{align*}
where we have performed the change of variable from $w$ to $z:=\frac{v-w}{\varepsilon}$ in the integral. Expanding
$$ f_\varepsilon(v-\varepsilon z,t)=f_\varepsilon(v,t)-\varepsilon\partial_vf_\varepsilon(v,t)z+\frac{\varepsilon^2}{2}\partial^2_vf_\varepsilon(v,t)z^2
    -\frac{\varepsilon^3}{6}\partial^3_vf_\varepsilon(\tilde{v},t)z^3, $$
where $\tilde{v}:=(1-\theta)v+\theta(v-\varepsilon z)=v-\theta\varepsilon z$ for some $\theta\in [0,\,1]$, and recalling~\ref{ass:M.mass}--\ref{ass:M.3rd_mom} we further find
\begin{align*}
    \partial_t(f_\varepsilon-g) &= \bigl(r(v)-\rho_\varepsilon\bigr)(f_\varepsilon-g)-(\rho_\varepsilon-\varrho)g
        -\alpha\partial_v(f_\varepsilon-g)+\frac{\beta}{2}\partial^2_v(f_\varepsilon-g) \\
    &\phantom{=} +\frac{\alpha^2\varepsilon^2}{2}\partial^2_vf_\varepsilon
        -\frac{\varepsilon}{6}\int_\R z^3\cM(z;\,\alpha\varepsilon,\beta)\partial^3_vf_\varepsilon(\tilde{v},t)\,dz.
\end{align*}
We now multiply both sides by $2(f_\varepsilon-g)$ and integrate w.r.t. $v\in\R$ to obtain
\begin{align*}
    \frac{d}{dt}\norm{(f_\varepsilon-g)(t)}{L^2}^2 &= 2\int_\R(r(v)-\rho_\varepsilon)(f_\varepsilon(v,t)-g(v,t))^2\,dv \\
    &\phantom{=} -2(\rho_\varepsilon-\varrho)\int_\R g(v,t)(f_\varepsilon(v,t)-g(v,t))\,dv \\
    &\phantom{=} -2\alpha\int_\R\partial_v(f_\varepsilon(v,t)-g(v,t))(f_\varepsilon(v,t)-g(v,t))\,dv \\
    &\phantom{=} +\beta\int_\R\partial^2_v(f_\varepsilon(v,t)-g(v,t))(f_\varepsilon(v,t)-g(v,t))\,dv \\
    &\phantom{=} +\alpha^2\varepsilon^2\int_\R\partial^2_vf_\varepsilon(v,t)(f_\varepsilon(v,t)-g(v,t))\,dv \\
    &\phantom{=} -\frac{\varepsilon}{3}\int_\R\int_\R z^3\cM(z;\,\alpha\varepsilon,\beta)\partial^3_vf_\varepsilon(\tilde{v},t)
        (f_\varepsilon(v,t)-g(v,t))\,dz\,dv.
\intertext{Next, we notice that $2(f_\varepsilon-g)\partial_v(f_\varepsilon-g)=\partial_v(f_\varepsilon-g)^2$ and we exploit Cauchy--Schwarz inequality and integration-by-parts to further elaborate this as}
    &\leq 2\norm{r}{\infty}\norm{(f_\varepsilon-g)(t)}{L^2}^2+2\abs{(\rho_\varepsilon-\varrho)(t)}\cdot\norm{g(t)}{L^2}\norm{(f_\varepsilon-g)(t)}{L^2} \\
    &\phantom{\leq} -\alpha\Bigl((f_\varepsilon(v,t)-g(v,t))^2\Bigr\vert_{v=-\infty}^{v=+\infty} \\
    &\phantom{\leq} +\beta\Bigl(\partial_v(f_\varepsilon(v,t)-g(v,t))(f_\varepsilon(v,t)-g(v,t))\Bigr\vert_{v=-\infty}^{v=+\infty} \\
    &\phantom{\leq} \qquad -\beta\int_\R\bigl(\partial_v(f_\varepsilon(v,t)-g(v,t)\bigr)^2dv \\
    &\phantom{\leq} +2\varepsilon\left(\alpha\varepsilon\norm{\partial^2_vf_\varepsilon(t)}{\infty}
        +\frac{1}{3}\norm{\partial^3_vf_\varepsilon(t)}{\infty}\bigl(\beta^{3/2}\ave{\abs{\tilde{Z}^3}}+1\bigr)\right)(\rho_\varepsilon+\varrho)(t),
\intertext{where $\tilde{Z}$ is the random variable introduced in~\eqref{eq:aveZ3}. Applying Proposition~\ref{prop:bc_feps} together with the boundary conditions~\eqref{eq:bc_g} along with Propositions~\ref{prop:bound.rho_eps},~\ref{prop:bound.varrho},~\ref{prop:L2_g},~\ref{prop:rhoeps-varrho}, we arrive at}
    &\leq 2\norm{r}{\infty}\norm{(f_\varepsilon-g)(t)}{L^2}^2 \\
    &\phantom{\leq} +2\norm{r}{L^2}\norm{f^0}{L^2}e^{3\norm{r}{\infty}t}\norm{(f_\varepsilon-g)(t)}{L^2}\int_0^t\norm{(f_\varepsilon-g)(\tau)}{L^2}\,d\tau \\
    &\phantom{\leq} +4\varepsilon\norm{r}{\infty}\left(\alpha\varepsilon\norm{\partial^2_vf_\varepsilon(t)}{\infty}
        +\frac{1}{3}\bigl(\beta^{3/2}\ave{\abs{\tilde{Z}^3}}+1\bigr)\norm{\partial^3_vf_\varepsilon(t)}{\infty}\right).
\end{align*}

Next, we integrate both sides on $[0,\,t]$, $t>0$, and use the fact that $f_\varepsilon(v,0)=g(v,0)=f^0(v)$ to deduce
\begin{align*}
    \norm{(f_\varepsilon-g)(t)}{L^2}^2 &\leq 2\norm{r}{\infty}\int_0^t\norm{(f_\varepsilon-g)(s)}{L^2}^2\,ds \\
    &\phantom{\leq} +2\norm{r}{L^2}\norm{f^0}{L^2}\int_0^te^{3\norm{r}{\infty}s}\norm{(f_\varepsilon-g)(s)}{L^2}
        \int_0^s\norm{(f_\varepsilon-g)(\tau)}{L^2}\,d\tau\,ds \\
    &\phantom{\leq} +4\varepsilon\norm{r}{\infty}\left(\alpha\varepsilon\int_0^t\norm{\partial^2_vf_\varepsilon(s)}{\infty}\,ds
        +\frac{1}{3}\bigl(\beta^{3/2}\ave{\abs{\tilde{Z}^3}}+1\bigr)\int_0^t\norm{\partial^3_vf_\varepsilon(s)}{\infty}\,ds\right) \\
    &\leq 2\norm{r}{\infty}\int_0^t\norm{(f_\varepsilon-g)(s)}{L^2}^2\,ds \\
    &\phantom{\leq} +2\norm{r}{L^2}\norm{f^0}{L^2}e^{3\norm{r}{\infty}t}\left(\int_0^t\norm{(f_\varepsilon-g)(s)}{L^2}\,ds\right)^2 \\
    &\phantom{\leq} +4\varepsilon\norm{r}{\infty}\left(\alpha\varepsilon\int_0^t\norm{\partial^2_vf_\varepsilon(s)}{\infty}\,ds
        +\frac{1}{3}\bigl(\beta^{3/2}\ave{\abs{\tilde{Z}^3}}+1\bigr)\int_0^t\norm{\partial^3_vf_\varepsilon(s)}{\infty}\,ds\right),
\intertext{from which, applying Cauchy--Schwarz inequality to the second term on the right-hand side,}
    &\leq 2\left(\norm{r}{\infty}+\norm{r}{L^2}\norm{f^0}{L^2}te^{3\norm{r}{\infty}t}\right)\int_0^t\norm{(f_\varepsilon-g)(s)}{L^2}^2\,ds \\
    &\phantom{\leq} +4\varepsilon\norm{r}{\infty}\left(\alpha\varepsilon\int_0^t\norm{\partial^2_vf_\varepsilon(s)}{\infty}\,ds
        +\frac{1}{3}\bigl(\beta^{3/2}\ave{\abs{\tilde{Z}^3}}+1\bigr)\int_0^t\norm{\partial^3_vf_\varepsilon(s)}{\infty}\,ds\right).
\intertext{Proposition~\ref{prop:Linf_estimates} implies that $\norm{\partial^2_vf_\varepsilon(t)}{\infty}$ and $\norm{\partial^3_vf_\varepsilon(t)}{\infty}$ are $\varepsilon$-uniformly bounded in $(0,\,T]$ by a constant $C_T>0$, thus}
    &\leq 2\left(\norm{r}{\infty}+\norm{r}{L^2}\norm{f^0}{L^2}te^{3\norm{r}{\infty}t}\right)\int_0^t\norm{(f_\varepsilon-g)(s)}{L^2}^2\,ds \\
    &\phantom{\leq} +4\varepsilon\norm{r}{\infty}C_T\left(\alpha\varepsilon+\frac{1}{3}\bigl(\beta^{3/2}\ave{\abs{\tilde{Z}^3}}+1\bigr)\right)t.
\end{align*}

Finally, Lemma~\ref{lemma:Gronw_gen} (see Appendix~\ref{appendixB}) applied to $u(t)=\norm{(f_\varepsilon-g)(t)}{L^2}^2$ with
$$ a(t)=4\varepsilon\norm{r}{\infty}C_T\left(\alpha\varepsilon+\frac{1}{3}\bigl(\beta^{3/2}\ave{\abs{\tilde{Z}^3}}+1\bigr)\right)t,
    \qquad b(t)=2\left(\norm{r}{\infty}+\norm{r}{L^2}\norm{f^0}{L^2}te^{3\norm{r}{\infty}t}\right), $$
which are non-decreasing non-negative functions in $[0,\,T]$, yields
$$ \norm{(f_\varepsilon-g)(t)}{L^2}^2\leq
    4\varepsilon\norm{r}{\infty}C_T\left(\alpha\varepsilon+\frac{1}{3}\bigl(\beta^{3/2}\ave{\abs{\tilde{Z}^3}}+1\bigr)\right)te^{b(t)t}
        \xrightarrow{\varepsilon\to 0^+} 0,$$
which concludes the proof.
\end{proof}

\subsubsection{Convergence of the statistical moments}
To grasp the big picture of a multi-agent system, one usually refers to aggregate macroscopic quantities represented by the statistical moments of the distribution function. We recall, in particular, the most significant low-order ones:
\begin{itemize}
\item $\rho_\varepsilon(t):=\displaystyle{\int_\R}f_\varepsilon(v,t)\,dv$, i.e. the density of the agents at time $t>0$ (zeroth order moment),
\item $p_\varepsilon(t):=\displaystyle{\int_\R}vf_\varepsilon(v,t)\,dv$, i.e. the phenotypic ``momentum'' at time $t>0$ (first order moment),
\item $E_\varepsilon(t):=\displaystyle{\int_\R}v^2f_\varepsilon(v,t)\,dv$, i.e. the phenotypic ``bulk energy'' at time $t>0$ (second order moment).
\end{itemize}

In this section, we show that in the quasi-invariant limit $\varepsilon\to 0^+$ these quantities converge to the corresponding moments of the limit phenotype distribution function $g$:
$$ \varrho(t):=\int_\R g(v,t)\,dv, \qquad p(t):=\int_\R vg(v,t)\,dv, \qquad E(t):=\int_\R v^2g(v,t)\,dv, $$
thereby establishing that the non-local PDE model~\eqref{eq:PDE.strong} inherits in the limit also the exact main statistical moments of the agent distribution. For this, besides~\eqref{eq:r,f0} we shall need the following additional assumption on the function $r$:
\begin{equation}
    v^2r\in L^2(\R).
    \label{eq:v^2r}
\end{equation}

The first result is a straightforward consequence of the theory developed in Section~\ref{sect:L2_convergence}.
\begin{theorem}[Convergence of the density] \label{theo:conv.density}
Under~\eqref{eq:r,f0},~\eqref{eq:init_cond}, $\rho_\varepsilon(t)\to\varrho(t)$ for every $t\in (0,\,T]$ and arbitrary $T>0$ as $\varepsilon\to 0^+$.
\end{theorem}
\begin{proof}
The thesis follows from the fact that Proposition~\ref{prop:rhoeps-varrho} and the proof of Theorem~\ref{theo:feps_to_g} imply jointly\footnote{The symbol $\lesssim$ indicates that there exists a constant $C>0$, independent of $\varepsilon$ and $t$, and the specific value of which is not relevant, such that $\abs{\rho_\varepsilon(t)-\varrho(t)}\leq C\sqrt{\varepsilon}$ for every $t\in (0,\,T]$. From the proof of Theorem~\ref{theo:feps_to_g}, it can be seen that such a constant does depend on $T$.} $\abs{\rho_\varepsilon(t)-\varrho(t)}\lesssim \sqrt{\varepsilon}$ for all $t\in (0,\,T]$.
\end{proof}

We proceed now with the other two moments.
\begin{theorem}[Convergence of the phenotypic momentum and bulk energy] \label{theo:conv.mean_energy}
Under~\eqref{eq:r,f0},~\eqref{eq:init_cond}, and~\eqref{eq:v^2r}, we have $p_\varepsilon(t)\to p(t)$ and $E_\varepsilon(t)\to E(t)$ for every $t\in (0,\,T]$ and arbitrary $T>0$ as $\varepsilon\to 0^+$.
\end{theorem}
\begin{proof}
\begin{enumerate}[label=(\roman*)]
\item We begin with the convergence of the phenotypic momentum. We observe preliminarily that from~\eqref{eq:v^2r} it follows also $vr\in L^2(\R)$; in fact
\begin{align*}
    \norm{vr}{L^2}^2=\int_\R v^2r^2(v)\,dv &= \int_{\{\abs{v}\leq 1\}}v^2r^2(v)\,dv+\int_{\{\abs{v}>1\}}v^2r^2(v)\,dv,
\intertext{and, therefore, since $v^2\leq 1$ for $\abs{v}\leq 1$ and $v^2<v^4$ for $\abs{v}>1$,}
    &\leq \int_{\{\abs{v}\leq 1\}}r^2(v)\,dv+\int_{\{\abs{v}>1\}}v^4r^2(v)\,dv
        \leq\norm{r}{L^2}^2+\norm{v^2r}{L^2}^2<+\infty.
\end{align*}

Letting now $\varphi(v)=v$ in~\eqref{eq:IDE.weak_scaled} with~\eqref{eq:mu.scaling} and recalling~\eqref{eq:moments_Meps}, we obtain
$$ \frac{dp_\varepsilon}{dt}=\int_\R vr(v)f_\varepsilon(v,t)\,dv-\rho_\varepsilon p_\varepsilon+\alpha\rho_\varepsilon; $$
likewise, substituting $\varphi(v)=v$ into~\eqref{eq:PDE.weak} we find
\begin{equation}
    \frac{dp}{dt}=\int_\R vr(v)g(v,t)\,dv-\varrho p+\alpha\varrho.
    \label{eq:varrho_barv}
\end{equation}
Subtracting these two equations produces
$$ \frac{d}{dt}(p_\varepsilon-p)=\int_\R vr(v)\bigl(f_\varepsilon(v,t)-g(v,t)\bigr)\,dv
    -\rho_\varepsilon(p_\varepsilon-p)+(\alpha-p)(\rho_\varepsilon-\varrho), $$
and thus
\begin{equation}
    \frac{d}{dt}\abs{p_\varepsilon-p}\leq\norm{vr}{L^2}\norm{(f_\varepsilon-g)(t)}{L^2}
        +\norm{r}{\infty}\abs{p_\varepsilon-p}+(\abs{p}+\abs{\alpha})\abs{\rho_\varepsilon-\varrho}.
    \label{eq:diff.rho_barv}
\end{equation}

Notice that from~\eqref{eq:varrho_barv} it follows that $p$ is bounded on every bounded time interval; in fact
$$ \frac{d}{dt}\abs{p}\leq\norm{vr}{L^2}\norm{g(t)}{L^2}+\norm{r}{\infty}\abs{p}+\abs{\alpha}\cdot\norm{r}{\infty}, $$
and thus, by Gr\"{o}nwall's inequality after bounding $\norm{g(t)}{L^2}$ by means of Proposition~\ref{prop:L2_g}, $\abs{p(t)}$ is bounded by a non-negative non-decreasing function, say $h=h(t)$. Using this in~\eqref{eq:diff.rho_barv} after integrating on $[0,\,t]$, $t>0$, yields
\begin{align*}
    \abs{(p_\varepsilon-p)(t)} &\leq \norm{rv}{L^2}\int_0^t\norm{(f_\varepsilon-g)(s)}{L^2}\,ds
        +\int_0^t(h(s)+\abs{\alpha})\abs{(\rho_\varepsilon-\varrho)(s)}\,ds \\
    &\phantom{\leq} +\norm{r}{\infty}\int_0^t\abs{(p_\varepsilon-p)(s)}\,ds,
\end{align*}
where we have taken into account that $p_\varepsilon(0)=p(0)$ because of~\eqref{eq:init_cond}.

At this stage, Lemma~\ref{lemma:Gronw_Tosc} (see Appendix~\ref{appendixB}) applied to the function $u(t)=\abs{(p_\varepsilon-p)(t)}$ with
$$ a(t)=\norm{rv}{L^2}\int_0^t\norm{(f_\varepsilon-g)(s)}{L^2}\,ds+\int_0^t(h(s)+\abs{\alpha})\abs{(\rho_\varepsilon-\varrho)(s)}\,ds,
    \qquad b(t)=\norm{r}{\infty} $$
produces
$$ \abs{(p_\varepsilon-p)(t)}\leq\int_0^t\bigl(\norm{vr}{L^2}\norm{(f_\varepsilon-g)(s)}{L^2}
    +(h(s)+\abs{\alpha})\abs{(\rho_\varepsilon-\varrho)(s)}\bigr)e^{\norm{r}{\infty}(t-s)}\,ds. $$
The thesis follows then passing the limit $\varepsilon\to 0^+$ through the integral by dominated convergence -- notice that, owing to Propositions~\ref{prop:bound.rho_eps},~\ref{prop:bound.varrho},~\ref{prop:L2_feps},~\ref{prop:L2_g}, the integrand is bounded by the integrable mapping $t\mapsto 2\norm{vr}{L^2}\norm{f^0}{L^2}e^{\norm{r}{\infty}t}+2(\abs{\alpha}+h(s))\norm{r}{\infty}$ -- and invoking Theorems~\ref{theo:feps_to_g},~\ref{theo:conv.density}.

\item As for the convergence of the phenotypic bulk energy, we proceed analogously to the previous step by substituting $\varphi(v)=v^2$ into~\eqref{eq:IDE.weak_scaled} with~\eqref{eq:mu.scaling} and into~\eqref{eq:PDE.weak} to find, respectively,
\begin{align*}
    \frac{dE_\varepsilon}{dt} &= \int_\R v^2r(v)f_\varepsilon(v,t)\,dv-\rho_\varepsilon E_\varepsilon+2\alpha p_\varepsilon
        +(\beta+\alpha^2
        \varepsilon^2)\rho_\varepsilon, \\
    \frac{dE}{dt} &= \int_\R v^2r(v)g(v,t)\,dv-\varrho E+2\alpha p+\beta\varrho.
\end{align*}
Subtracting these equations yields
\begin{align*}
    \frac{d}{dt}\abs{E_\varepsilon-E} &\leq \norm{v^2r}{L^2}\norm{(f_\varepsilon-g)(t)}{L^2} \\
    &\phantom{\leq} +\norm{r}{\infty}\abs{E_\varepsilon-E}+(\abs{E}+\beta)\abs{\rho_\varepsilon-\varrho} \\
    &\phantom{\leq} +2\abs{\alpha}\cdot\abs{p_\varepsilon-p}+\alpha^2\norm{r}{\infty}\varepsilon^2.
\end{align*}
In particular, the equation for $E$ implies that the latter is bounded by a non-negative non-decreasing function, say $k=k(t)$, because
$$ \frac{d}{dt}\abs{E}\leq\norm{v^2r}{L^2}\norm{g(t)}{L^2}+\norm{r}{\infty}\abs{E}+2\abs{\alpha}\cdot\abs{p}+\beta\norm{r}{\infty}, $$
where $\norm{g(t)}{L^2}$ can be bounded by means of Proposition~\ref{prop:L2_g}, while $\abs{p}$ is bounded by the function $h$ introduced in the previous step.

On the whole, we then have 
\begin{align*}
    \abs{(E_\varepsilon-E)(t)} &\leq \norm{v^2r}{L^2}\int_0^t\norm{(f_\varepsilon-g)(s)}{L^2}\,ds
        +\int_0^t(k(s)+\beta)\abs{(\rho_\varepsilon-\varrho)(s)}\,ds \\
    &\phantom{\leq} +2\abs{\alpha}\int_0^t\abs{(p_\varepsilon-p)(s)}\,ds+\alpha^2\norm{r}{\infty}\varepsilon^2t \\
    &\phantom{\leq} +\norm{r}{\infty}\int_0^t\abs{(E_\varepsilon-E)(s)}\,ds,
\end{align*}
which, invoking Lemma~\ref{lemma:Gronw_Tosc} (see Appendix~\ref{appendixB}) for the function $u(t)=\abs{(E_\varepsilon-E)(t)}$ with
\begin{align*}
    a(t) &= \norm{v^2r}{L^2}\int_0^t\norm{(f_\varepsilon-g)(s)}{L^2}\,ds+\int_0^t(k(s)+\beta)\abs{(\rho_\varepsilon-\varrho)(s)}\,ds \\
    &\phantom{=} +2\abs{\alpha}\int_0^t\abs{(p_\varepsilon-p)(s)}\,ds+\alpha^2\norm{r}{\infty}\varepsilon^2t, \\
    b(t) &= \norm{r}{\infty},
\end{align*}
implies
\begin{align*}
    \abs{(E_\varepsilon-E)(t)} &\leq \int_0^t\bigl(\norm{v^2r}{L^2}\norm{(f_\varepsilon-g)(s)}{L^2}
        +(k(s)+\beta)\abs{(\rho_\varepsilon-\varrho)(s)} \\
    &\phantom{\leq\int_0^t} +2\abs{\alpha}\cdot\abs{(p_\varepsilon-p)(s)}+\alpha^2\norm{r}{\infty}\varepsilon^2s\bigr)
        e^{\norm{r}{\infty}(t-s)}\,ds.
\end{align*}
The thesis follows then from Theorems~\ref{theo:feps_to_g},\,\ref{theo:conv.density} along with the result proved in the previous step, passing the limit $\varepsilon\to 0^+$ through the integral by dominated convergence. \qedhere
\end{enumerate}
\end{proof}

\begin{remark}
From the statistical moments above one can compute two other aggregate quantities of interest, such as the mean phenotype:
$$ \bar{v}_\varepsilon:=\frac{p_\varepsilon}{\rho_\varepsilon} $$
and the variance of the phenotype distribution:
$$ \sigma_\varepsilon^2:=\frac{E_\varepsilon}{\rho_\varepsilon}-\bar{v}_\varepsilon^2. $$
In the quasi-invariant regime, the corresponding quantities $\bar{v}$, $\sigma^2$ are defined analogously using $\varrho$, $p$, and $E$.
\end{remark}

\section{Main results of numerical simulations}
\label{sect:num}
\begin{table}[t]
\caption{Parameters used in the numerical simulations of Figures~\ref{fig:test_alpha_neg},~\ref{fig:test_alpha_null},~\ref{fig:test_alpha_pos}}
\label{tab:param}
\begin{center}
\begin{tabular}{|l|l|l|}
    \hline
    Parameter & Description & Value \\
    \hline\hline
    $N$ & Number of agents in MC simulations & $10^5$ \\
    $\Delta{v}$ & Discretisation step of phenotype domain & $2.5\cdot 10^{-2}$ \\
    $\Delta{t}$ & Discretisation step of time domain & $10^{-4}$ \\
    $T$ & Final time of simulations & $10$ \\
    & Phenotype domain & $[-15,\,15]$ \\
    $R$ & Radius of the mollifier $\psi$ & $5$ \\
    $\delta$ & Amplitude of the transition layer of the mollifier $\psi$ & $0.5$ \\
    $v_m$ & Fittest trait (i.e. maximum point of the function $r$) & $1.5$ \\
    $\alpha$ & Drift coefficient & $-0.3$, $0$, $0.3$ \\
    $\beta$ & Diffusion coefficient & $0.4$ \\
    $\varepsilon$ & Scaling parameter for the quasi-invariant regime & $1$, $10^{-1/2}$, $10^{-1}$ \\
    \hline
\end{tabular}
\end{center}
\end{table}

We now illustrate, by means of a sample of results of numerical simulations, the theoretical results obtained in the previous sections. In more detail, we present the results of numerical simulations carried out for:
\begin{itemize}
\item the stochastic agent-based model~\eqref{eq:3rand}, through a direct Monte Carlo method, cf.~\cite{pareschi2013BOOK};
\item the IDE model~\eqref{eq:IDE.scaled}, for three decreasing values of the scaling parameter $\varepsilon$, through an explicit-in-time Euler scheme complemented with a numerical approximation of the integral term based on the trapezium rule;
\item the limit non-local PDE model~\eqref{eq:PDE.strong},  through an explicit-in-time Euler scheme complemented with finite difference schemes for the approximation of the $v$-derivatives.
\end{itemize}
We compare and contrast these results in terms of distribution functions and low-order statistical moments. Specifically, the objective is to verify that there is quantitative agreement between the numerical solutions of the stochastic agent-based and IDE models for all the values of the scaling parameter, $\varepsilon$, thus corroborating the formal derivation presented in Section~\ref{sect:mesoscopic}. Moreover, we aim to show the convergence of the solution of the IDE model, and of its low-order statistical moments, to the solution of the non-local PDE model, and to the corresponding low-order moments, as the scaling parameter gets smaller, in order to provide numerical evidence for the quasi-invariant limit analytically studied in Section~\ref{sect:quasi-invariant}.

The parameter values employed in numerical simulations are summarised in Table~\ref{tab:param}.

As an initial condition, common to all simulations and to equations~\eqref{eq:IDE.scaled},~\eqref{eq:PDE.strong}, cf.~\eqref{eq:init_cond}, we prescribe
$$ f^0(v):=\frac{3}{10\sqrt{\pi}}e^{-v^2}, $$
namely a normal distribution with mass $\rho^0=0.3$, null mean and variance equal to $0.5$, which complies with assumptions~\eqref{eq:r,f0}. This is also the distribution from which the agents are initially sampled for the Monte Carlo method used for the stochastic agent-based model~\eqref{eq:3rand}.

\rev{We consider a Gaussian mutation kernel $\cM$ with mean $\alpha\varepsilon$ and variance $\beta$. The scaled kernel $M_\varepsilon$ used to carry out numerical simulations is then obtained in accordance with the formulation given in~\eqref{eq:M.scaling}. The values of $\alpha$, $\beta$, and $\varepsilon$ are those specified in Table~1.}

Concerning the net proliferation rate, $r$, we choose
$$ r(v):=\left(1-(v-v_m)^2\right)\psi(v). $$
The parabolic profile $1-(v-v_m)^2$ is maximum at $v=v_m$ with $r(v_m)=1$, cf.~\eqref{ass:r_with_sup}, which represents the so-called \textit{fittest trait}, i.e. the phenotype corresponding to the maximum net proliferation rate. The factor $\psi:\R\to [0,\,1]$ is a $C^\infty$-mollifier, which is needed to adjust the trend of $r$ at infinity, in compliance with assumptions~\eqref{eq:r,f0}, while allowing $r$ to coincide with the desired parabolic profile at finite points. In more detail, $\psi$ is defined by first introducing the functions
$$ \zeta(v):=\frac{1}{2}\left[1+\tanh{\left(\frac{2v}{1-v^2}\right)}\right]\chi_{(-1,\,1)}(v),
    \qquad \hat{\psi}(v):=1-\zeta\left(\frac{2(v-R)}{\delta}-1\right), $$
where $\chi_{(-1,\,1)}$ denotes the characteristic function of the interval $(-1,\,1)$, and then letting
$$  \psi(v):=
    \begin{cases}
        \hat{\psi}(-v) & \text{if } -R-\delta<v<-R \\
        1 & \text{if } -R\leq v\leq R \\
        \hat{\psi}(v) & \text{if } R<v<R+\delta \\
        0 & \text{otherwise},
    \end{cases}
$$
where $\delta,\,R>0$ are real parameters, cf. Table~\ref{tab:param}. Notice that $\psi\equiv 1$, and thus $r(v)=1-(v-v_m)^2$, for $v\in [-R,\,R]$. In contrast, in the two layers $(-R-\delta,\,-R)$ and $(R,\,R+\delta)$ of amplitude $\delta$ the function $\hat{\psi}$ produces a $C^\infty$-transition of $\psi$ from $0$ to $1$ and from $1$ to $0$, respectively.

\begin{figure}[t]
    \centering
    \includegraphics[width=\linewidth]{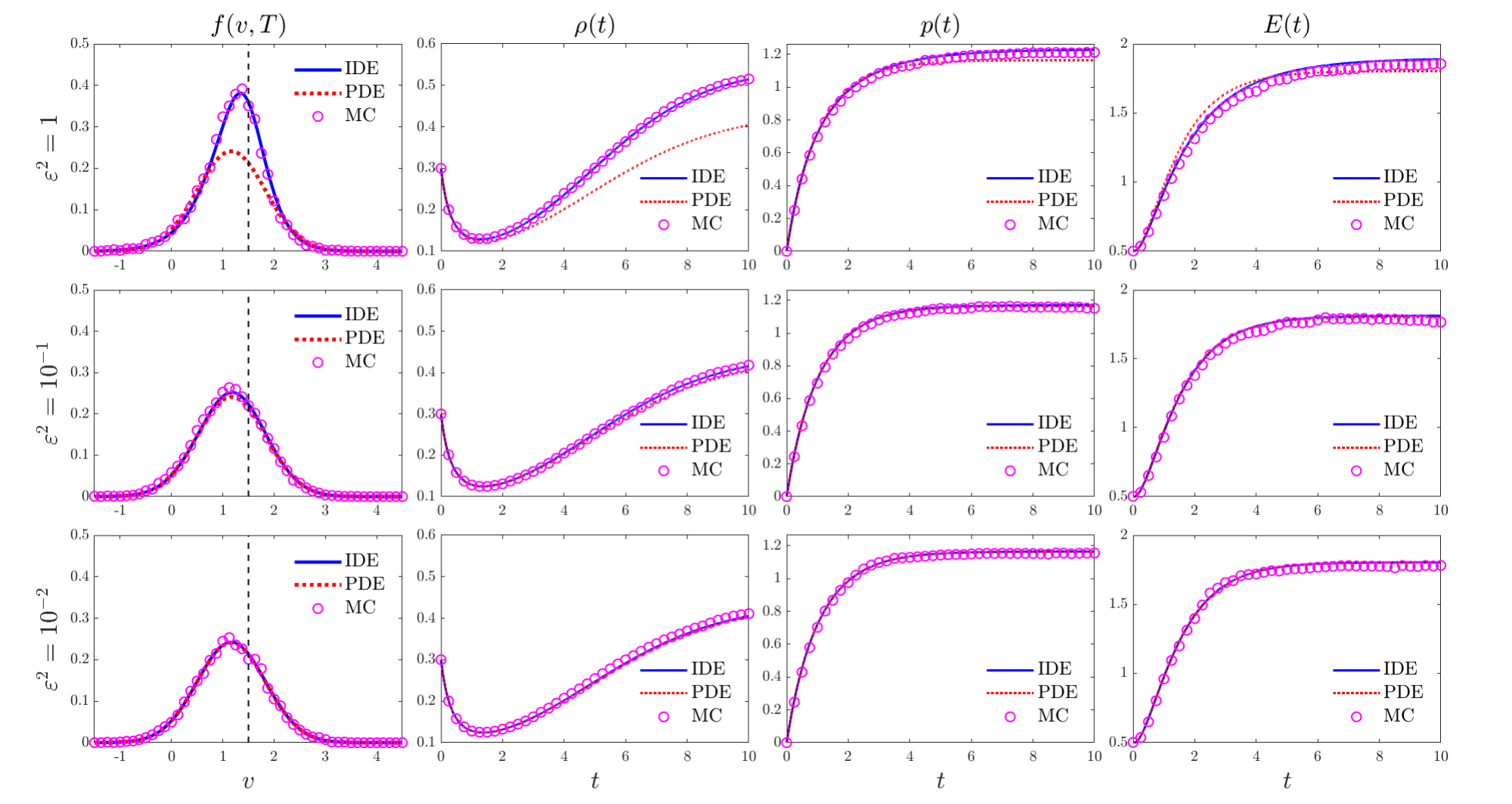}
    \caption{Results of numerical simulations of the stochastic agent-based model~\eqref{eq:3rand} (MC), the IDE model~\eqref{eq:IDE.scaled} (IDE), and the limit non-local PDE model~\eqref{eq:PDE.strong} (PDE) for the drift coefficient $\alpha<0$. The dashed vertical line in the panels of the first column highlights the fittest trait $v_m$}
    \label{fig:test_alpha_neg}
\end{figure}
\begin{figure}[t]
    \centering
    \includegraphics[width=\linewidth]{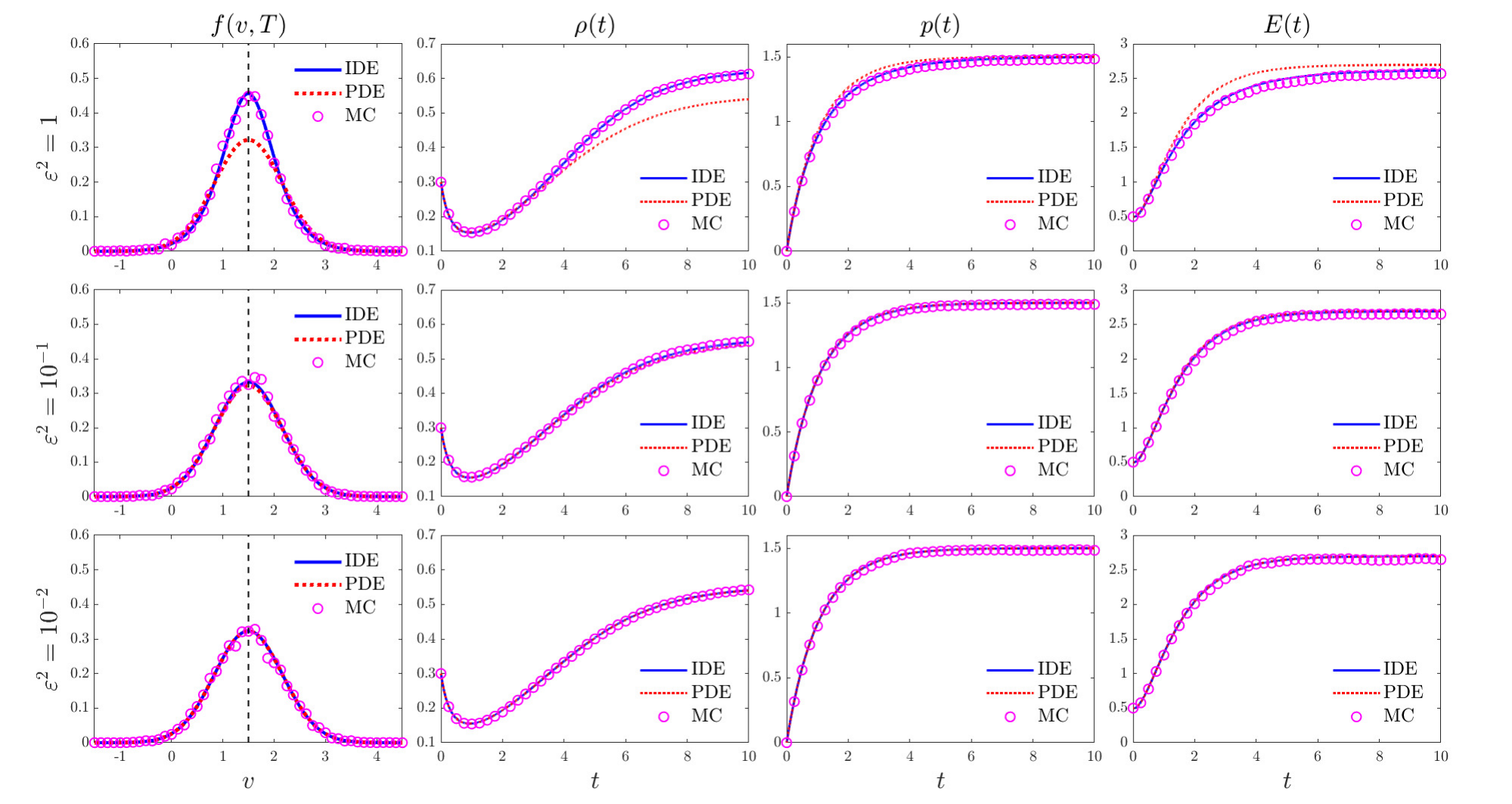}
    \caption{Results of numerical simulations of the stochastic agent-based model~\eqref{eq:3rand} (MC), the IDE model~\eqref{eq:IDE.scaled} (IDE), and the limit non-local PDE model~\eqref{eq:PDE.strong} (PDE) for the drift coefficient $\alpha=0$. The dashed vertical line in the panels of the first column highlights the fittest trait $v_m$}
    \label{fig:test_alpha_null}
\end{figure}
\begin{figure}[!ht]
    \centering
    \includegraphics[width=\linewidth]{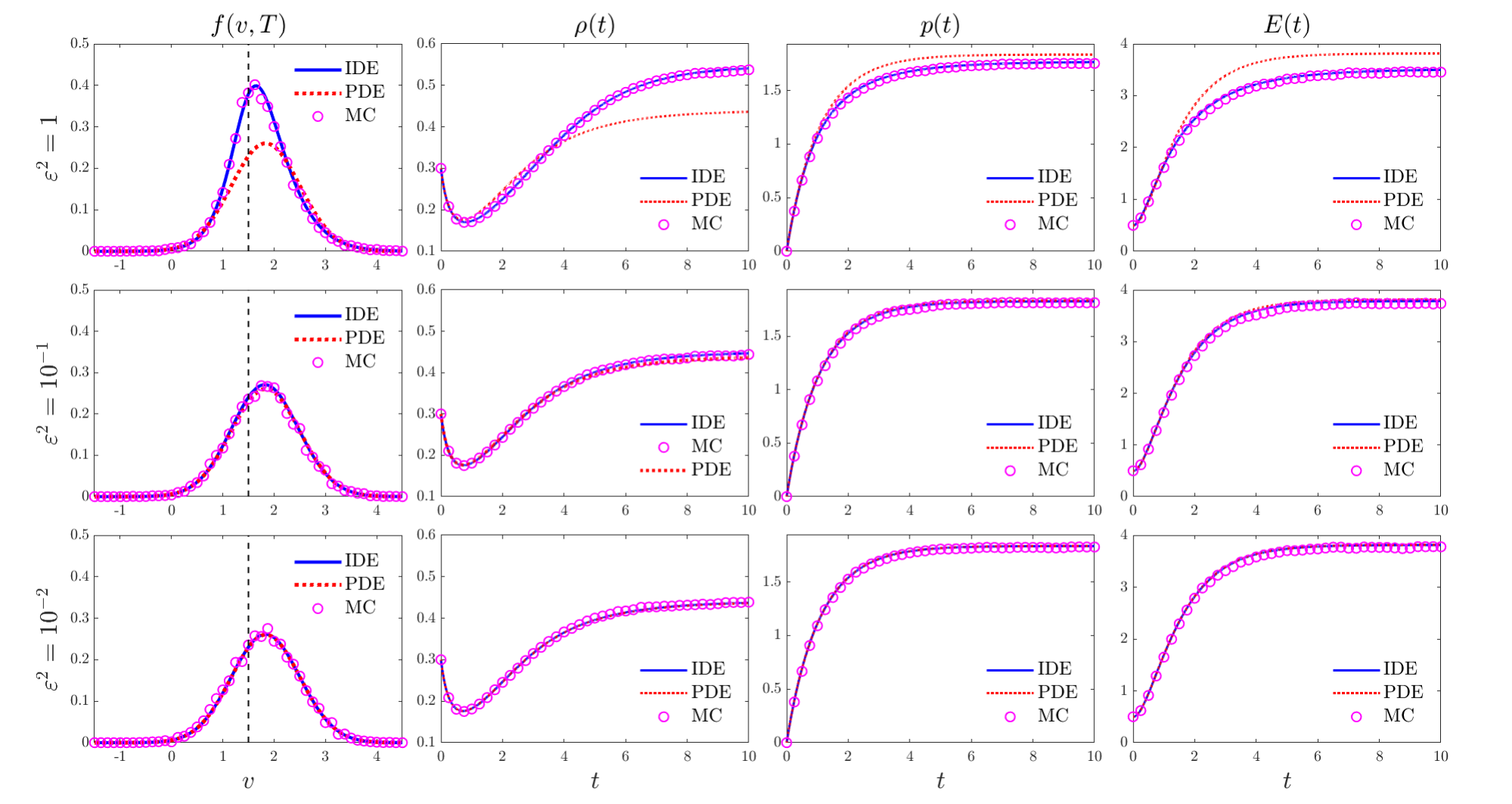}
    \caption{Results of numerical simulations of the stochastic agent-based model~\eqref{eq:3rand} (MC), the IDE model~\eqref{eq:IDE.scaled} (IDE), and the limit non-local PDE model~\eqref{eq:PDE.strong} (PDE) for the drift coefficient $\alpha>0$. The dashed vertical line in the panels of the first column indicates the fittest trait $v_m$}
    \label{fig:test_alpha_pos}
\end{figure}

Figures~\ref{fig:test_alpha_neg},~\ref{fig:test_alpha_null},~\ref{fig:test_alpha_pos} display the results of numerical simulations for three different values of the drift coefficient, $\alpha$, that is, $\alpha<0$ (Figure~\ref{fig:test_alpha_neg}), $\alpha=0$ (Figure~\ref{fig:test_alpha_null}) and $\alpha>0$ (Figure~\ref{fig:test_alpha_pos}) -- see Table~\ref{tab:param} for further details on the parameter values. In summary, the rows of each figure display, from left to right, the phenotype distribution function at the final time of simulations, $T=10$, and the time evolution from $t=0$ to $t=T$ of the density of agents, the phenotypic momentum, and the phenotypic bulk energy, computed for the three models for a given value of the scaling parameter $\varepsilon$, which decreases, from row to row, from $\varepsilon=1$ through to $\varepsilon=\frac{1}{\sqrt{10}}$ to $\varepsilon=\frac{1}{10}$. In all cases, it is evident that, while there is excellent quantitative agreement between numerical solutions of the agent-based model~\eqref{eq:3rand} and the IDE models~\eqref{eq:IDE.scaled} at all times and for all values of $\varepsilon$, the agreement with the solution of the non-local PDE model~\eqref{eq:PDE.strong} improves as $\varepsilon$ decreases, as expected from the quasi-invariant asymptotics. Note that, for the employed simulation set-up, a good qualitative agreement with the solution of the non-local PDE model~\eqref{eq:PDE.strong} is already observed for a value of $\varepsilon$ as moderately small as $\frac{1}{\sqrt{10}}\approx 0.32$.

\section{Conclusions}
\label{sect:conclusions}
In this work, we have advanced and employed a derivation method for aggregate mathematical descriptions of evolutionary dynamics in phenotype-structured populations, where population members undergo proliferation, death, and phenotype changes. This method is rooted in kinetic theory approaches for mass-varying multi-agent systems and makes it possible to bridge, in a consistent manner, three types of models for phenotype-structured populations commonly found in the literature: agent-based models tracking the dynamics of single population members, wherein phenotype changes and proliferation and death are represented as stochastic particle processes; IDE models governing the dynamics of phenotype distribution functions, in which phenotype changes are described by an integral operator encoding a mutation kernel, whereas proliferation and death are encapsulated in a non-local reaction term; and non-local PDEs for phenotype distribution functions, in which phenotype changes are taken into account by an advection-diffusion term, while proliferation and death are still represented by a non-local reaction term.

Our approach incorporates, from the very beginning, that is, from the agent-based level of representation, non-conservative phenomena (i.e. proliferation and death of population members). This has required special care in passing to the IDE model through the kinetic paradigm for multi-agent systems, \rev{since changes in the number of agents, and thus variations in the mass of the system}, is non-routine within such a paradigm. Moreover, this approach has shed light on the mutual relations among the aforementioned alternative types of models, which in the literature are usually addressed separately. In particular, we have shown that the IDE model is the aggregate counterpart of the agent-based model in every regime of the parameters driving the evolutionary dynamics of the population, whereas the non-local PDE model provides a faithful representation of the underlying evolutionary dynamics only in the regime of small and frequent phenotype changes.

A natural extension of this work would be to generalise the proposed approach to cases where, in addition to undergoing proliferation, death, and phenotype changes, population members move across space. In recent decades, a range of probabilistic methods~\cite{andrade2019local,champagnat2007invasion,fontbona2015non,leman2016convergence} and formal limiting procedures~\cite{freingruber2024inprogress, macfarlane2022individual,lorenzi2023derivation} have been developed for transitioning between stochastic agent-based models and deterministic continuum models which describe simultaneously the spatial spread and evolutionary dynamics of phenotype-structured populations. Yet, although phenotype-structuring variables have also been incorporated into kinetic equations \rev{(see e.g. \cite{engwer2015glioma,erban2004individual,lorenzi2024phenotype,perthame2,perthame1})}, we are unaware of any such derivations rooted in kinetic approaches of the type advanced here. In this regard, we envisage appropriate extensions of the limiting procedures employed in this article to enable harnessing a range of methods and techniques from across agent-based, IDE, and non-local PDE models, which would enhance further progress in the mathematical formalisation and analysis of the evolutionary dynamics of phenotype-structured populations.

\appendix

\section*{Appendix}

\section{The case of constant~\texorpdfstring{$\boldsymbol{r}$}{}}
\label{app:const_r}
As stated in Remark~\ref{rem:constant_r}, the theory developed in the body of the paper does not cover the case of constant $r$. For the sake of theoretical speculation and completeness, in this appendix we provide the technical details needed to also address the simpler quasi-invariant limit in the case of constant $r\leq 1$.

In the following, we keep assumptions~\eqref{eq:r,f0},~\eqref{eq:init_cond}, except that we drop the request $r\in L^2(\R)$.

The main simplification that occurs with a constant $r$ is that the density $\rho_\varepsilon$ obeys a self-consistent equation independent of $\varepsilon$. In fact, substituting $\varphi(v)=1$ into~\eqref{eq:IDE.weak_scaled} we obtain
$$ \frac{d\rho_\varepsilon}{dt}=(r-\rho_\varepsilon)\rho_\varepsilon, $$
which is a logistic equation admitting the $\varepsilon$-free solution
$$ \rho_\varepsilon(t)=\frac{r\rho^0}{(r-\rho^0)e^{-rt}+\rho^0}. $$
If $0\leq\rho^0\leq r$, the result of Proposition~\ref{prop:bound.rho_eps} still holds qualitatively and we have, in particular, $0\leq\rho_\varepsilon(t)\leq r$ for all $t>0$.

The same reasoning applied to~\eqref{eq:PDE.weak} with $\varphi(v)=1$ shows that the density $\varrho$ fulfils the very same equation as $\rho_\varepsilon$, that is, 
$$ \frac{d\varrho}{dt}=(r-\varrho)\varrho, $$
and then $\rho_\varepsilon(t)=\varrho(t)$ for every $\varepsilon>0$ and every $t\geq 0$. In practice, there exists a unique density, say $\rho$, which is independent of $\varepsilon$ and is the same for both phenotype distribution functions $f_\varepsilon$, $g$. Therefore, there is no need to pass to the limit $\varepsilon\to 0^+$ in the density.

On the other hand, Propositions~\ref{prop:non-neg.feps},~\ref{prop:non-neg.g},~\ref{prop:L2_feps},~\ref{prop:Linf_estimates},~\ref{prop:bc_feps},~\ref{prop:L2_g} still hold because none of their proofs uses the $L^2$-integrability of $r$, as can be ascertained by direct inspection. 

The consequence of these facts is that Theorem~\ref{theo:feps_to_g} also holds when $r$ is constant. In fact, the term $\abs{(\rho_\varepsilon-\varrho)(t)}$, which in the proof of the theorem carries the dependence on $\norm{r}{L^2}$, vanishes whereas the rest of the proof remains unchanged due to the boundedness of $\partial^2_vf_\varepsilon$ and $\partial^3_vf_\varepsilon$. Specifically, the $L^2$-estimate on $f_\varepsilon-g$ provided in the proof modifies as
$$ \norm{(f_\varepsilon-g)(t)}{L^2}^2\leq 2r\int_0^t\norm{(f_\varepsilon-g)(s)}{L^2}^2\,ds
    +4\varepsilon rC_T\left(\alpha\varepsilon+\frac{1}{3}(\beta^{3/2}\ave{\abs{\tilde{Z}}^3}+1)\right)t, $$
and then one obtains the convergence $f_\varepsilon(t)\to g(t)$ in $L^2(\R)$ for every $t>0$ as $\varepsilon\to 0^+$ by applying the standard Gr\"{o}nwall's inequality.

As far as the convergence of the statistical moments is concerned, we notice that also the phenotypic momentum of $f_\varepsilon$ does not depend on $\varepsilon$. In fact, substituting $\varphi(v)=v$ into~\eqref{eq:IDE.weak_scaled} with~\eqref{eq:mu.scaling} we find
$$ \frac{dp_\varepsilon}{dt}=(r-\rho)p_\varepsilon+\alpha\rho, $$
which is the very same equation satisfied by the phenotypic momentum $p$ of $g$, as it can be verified by substituting $\varphi(v)=v$ into~\eqref{eq:PDE.weak}. Thus, $p_\varepsilon(t)=p(t)$ for every $\varepsilon>0$ and every $t\geq 0$, which makes it unnecessary to pass to the limit $\varepsilon\to 0^+$. Conversely, the phenotypic bulk energies $E_\varepsilon$, $E$ satisfy
$$ \frac{dE_\varepsilon}{dt}=(r-\rho)E_\varepsilon+2\alpha p+(\beta+\alpha^2\varepsilon^2)\rho, \qquad
    \frac{dE}{dt}=(r-\rho)E+2\alpha p+\beta\rho, $$
whence
$$ \frac{d}{dt}(E_\varepsilon-E)=(r-\rho)(E_\varepsilon-E)+\alpha^2\varepsilon^2\rho $$
and, finally,
$$ \abs{(E_\varepsilon-E)(t)}\leq 2r\int_0^t\abs{(E_\varepsilon-E)(s)}\,ds+\alpha^2r\varepsilon^2t. $$
From here, Gr\"{o}nwall's inequality provides the convergence $E_\varepsilon(t)\to E(t)$ for all $t>0$ when $\varepsilon\to 0^+$, in particular without assumption~\eqref{eq:v^2r} which, with constant $r$, could not be met.

\section{Generalised Gr\"{o}nwall-type inequalities}
\label{appendixB}
In this section, we report two generalisations of the classical Gr\"{o}nwall's inequality, which we exploit for the development of our theory.

\begin{lemma} \label{lemma:Gronw_Tosc}
For $T>0$, let $a,\,b,\,u:[0,\,T]\to\R$ be real-valued continuous functions such that $a$ is differentiable in $(0,\,T)$, $b$ is non-negative in $[0,\,T]$ and also
$$ u(t)\leq a(t)+\int_0^tb(s)u(s)\,ds, \qquad \forall\,t\in [0,\,T]. $$
Then
$$ u(t)\leq a(0)e^{\int_0^tb(s)\,ds}+\int_0^ta'(s)e^{\int_s^tb(\tau)\,d\tau}\,ds, \qquad \forall\,t\in [0,\,T]. $$
\end{lemma}
\begin{proof}
Under the present hypotheses, the classical Gr\"{o}nwall's inequality implies (cf. e.g.~\cite{mitrinovic1991BOOK})
\begin{align*}
    u(t)\leq a(t)+\int_0^ta(s)b(s)e^{\int_s^tb(\tau)\,d\tau}\,ds &=
        a(t)-\int_0^ta(s)\frac{d}{ds}e^{\int_s^tb(\tau)\,d\tau}\,ds,
\intertext{and then, integrating by parts,}
    &= a(t)-\left(a(s)e^{\int_s^tb(\tau)\,d\tau}\right\vert_0^t+\int_0^ta'(s)e^{\int_s^tb(\tau)\,d\tau}\,ds \\
    &= a(0)e^{\int_0^tb(\tau)\,d\tau}+\int_0^ta'(s)e^{\int_s^tb(\tau)\,d\tau}\,ds,
\end{align*}
which gives the thesis.
\end{proof}

\begin{lemma} \label{lemma:Gronw_gen}
For $T>0$, let $a,\,b,\,u:[0,\,T]\to\R$ be real-valued continuous functions with $b$ non-negative and $a,\,b$ non-decreasing. Assume moreover that
\begin{equation}
    u(t)\leq a(t)+b(t)\int_0^tu(s)\,ds, \qquad \forall\,t\in [0,\,T].
    \label{hyp:Gronw1}
\end{equation}
Then
$$ u(t)\leq a(t)e^{b(t)t}, \qquad \forall\,t\in [0,\,T]. $$
\end{lemma}
\begin{remark}
This result is a variation of the classical Gr\"{o}nwall's inequality, in that the function $b$ is neither constant nor $s$-dependent within the integral on the right-hand side of~\eqref{hyp:Gronw1}.
\end{remark}
\begin{proof}[Proof of Lemma~\ref{lemma:Gronw_gen}]
We introduce the auxiliary function
$$ v(s):=e^{-\int_0^sb(\tau)\,d\tau}\int_0^su(\tau)\,d\tau, $$
which is such that
$$ v'(s)=e^{-\int_0^sb(\tau)\,d\tau}\left(u(s)-b(s)\int_0^su(\tau)\,d\tau\right)\leq a(s)e^{-\int_0^sb(\tau)\,d\tau}, $$
with the last inequality following from~\eqref{hyp:Gronw1}. Integrating on $[0,\,t]$, with $0<t\leq T$, and noting that $v(0)=0$ by definition, we find
$$ v(t)\leq\int_0^ta(s)e^{-\int_0^sb(\tau)\,d\tau}\,ds\leq a(t)\int_0^te^{-\int_0^sb(\tau)\,d\tau}\,ds, $$
where we have used that $a(s)\leq a(t)$ for all $s\leq t$, because $a$ is non-decreasing. Hence, recalling the definition of $v$,
\begin{align*}
    b(t)e^{\int_0^tb(\tau)\,d\tau}v(t) &= b(t)\int_0^tu(r)\,dt \\
    &\leq a(t)b(t)\int_0^te^{\int_s^tb(\tau)\,d\tau}\,ds \\
    &\leq a(t)b(t)\int_0^te^{b(t)(t-s)}\,ds,
\end{align*}
where we have used the fact that $b(r)\leq b(t)$ for all $r\leq t$, because also $b$ is non-decreasing. Then,
\begin{align*}
    b(t)\int_0^tu(\tau)\,d\tau &\leq a(t)e^{b(t)t}\int_0^tb(t)e^{-b(t)s}\,ds \\
    &= a(t)e^{b(t)t}\left(-e^{-b(t)s}\right\vert_0^t \\
    &= a(t)\left(e^{b(t)t}-1\right).
\end{align*}
Substituting this into~\eqref{hyp:Gronw1} yields the thesis.
\end{proof}

\bibliographystyle{plain}
\bibliography{biblio}
\end{document}